\documentclass[a4paper,11pt]{amsart}
\usepackage{amssymb,mathrsfs}
\usepackage{amsthm}
\usepackage{amsmath}
\usepackage{latexsym}
\usepackage{fancyhdr}
\usepackage{ascmac}
\usepackage{color}
\usepackage[all]{xy}
\usepackage{amscd}

\theoremstyle{plain}
\newtheorem{thm}{Theorem}[section]
\newtheorem{prop}[thm]{Proposition}
\newtheorem{cor}[thm]{Corollary}
\newtheorem{lem}[thm]{Lemma}
\newtheorem{dfn}[thm]{Definition}
\newtheorem{conj}[thm]{Conjecture}
\newtheorem{rmk}[thm]{Remark}
\makeatletter

\@addtoreset{equation}{section}
\makeatother

\newcommand{\bQ}{\overline{\mathbb{Q}}}

\newcommand{\bF}{\overline{\mathbb{F}}}

\newcommand{\C}{\mathbb{C}}
\newcommand{\R}{\mathbb{R}}
\newcommand{\Q}{\mathbb{Q}}

\newcommand{\Z}{\mathbb{Z}}

\newcommand{\F}{\mathbb{F}}

\newcommand{\lra}{\longrightarrow}

\newcommand{\A}{\mathbb{A}}

\newcommand{\vp}{\varphi}

\renewcommand{\O}{\mathcal{O}}

\newcommand{\br}{\overline{\rho}}
\newcommand{\tr}{{\rm tr}}
\newcommand{\diag}{{\rm diag}}
\newcommand{\ds}{\displaystyle}

\newcommand{\G}{\Gamma}
\newcommand{\cG}{\mathcal{G}}

\newcommand{\la}{\lambda}

\newcommand{\GL}{{\rm GL}}
\newcommand{\SL}{{\rm SL}}
\newcommand{\SU}{{\rm SU}}
\newcommand{\GSp}{{\rm GSp}}
\newcommand{\Sp}{{\rm Sp}}

\newcommand{\rr}{\overline{r}}

\newcommand{\bK}{\overline{K}}

\newcommand{\MD}{{\rm MD}}
\newcommand{\ord}{{\rm ord}}
\newcommand{\wT}{\widetilde{T}}
\newcommand{\wt}{\widetilde{t}}
\newcommand{\ve}{\varepsilon}
\newcommand{\wrho}{\widetilde{\rho}}

\title[The residual monodromy for the Dwork family]
{The residual monodromy for the Dwork family in even characteristic and its applications to Galois representations}
\author{Takuya Yamauchi}
\keywords{The residual monodromy, the Dwork family, Galois representations, Automorphy}
\thanks{}
\subjclass[2010]{11F, 11F33, 11F80}

\address{Takuya Yamauchi \\ 
Mathematical Inst. Tohoku Univ.\\
 6-3,Aoba, Aramaki, Aoba-Ku, Sendai 980-8578, JAPAN}
\email{takuya.yamauchi.c3@tohoku.ac.jp}

\begin{document}

\maketitle

\begin{abstract}
We study the residual monodromy representations 
associated to the Dwork family in characteristic two. 
Various applications involving 2-adic and mod 2 Galois representations are discussed.
Combining the author's previous work with Tsuzuki and recent results of Boxer, Calegari, Gee, and Pilloni, we also prove the automorphy of certain rank 4 symplectic motives over a totally real field, arising from the Dwork quintic family, under suitable conditions.
\end{abstract}

\tableofcontents

\section{Introduction}\label{intro}

In this paper, we study the residual image of the monodromy 
representations attached to the Dwork family and study its 
application to 2-adic or mod 2 Galois representations which takes the values in 
general symplectic groups. 
To explain our motivation, we begin by recalling the progress following the monumental work
of Harris Shepherd-Barron and Taylor \cite{HSBT} and Clozel, Harris, and Taylor \cite{CHT}. 

Let $F$ be a CM field or a totally real field with the maximal totally real subfield $F^+$ so that 
$F=F^+$ if $F$ is totally real. 
Fix an embedding of $F$ into an algebraic closure $\bQ$ of $\Q$ and let $p$ be a prime number.  
Fix an isomorphism $\bQ_p\simeq \C$ and embeddings $\bQ\lra \C$ 
and  $\iota=\iota_p:\bQ\lra\bQ_p$ compatible with it.
In \cite{CHT}, to study various automophy lifting theorems for Galois representations, the authors first 
associated each conjugate self dual, $n$-dimensional 
$p$-adic Galois representation $\rho:G_F:={\rm Gal}(\overline{F}/F)\lra \GL_n(\bQ_p)$
with a continuous representation $$r:G_{F^+}\lra \cG_n(\bQ_p)$$ 
by introducing a non-connected reductive group 
$\cG_n=(GL_n\times GL_1)\rtimes \{1,j\}$ over $\Z$ which facilitates Taylor-Wiles-type argument.
Note that $\cG_n$ is regarded as an algebraic model of the $L$-group of 
the general unitary group $GU(n)=GU(n)(F/F^+)$ over $F^+$ whose derived group is anisotropic at each infinite place of $F^+$. 
The deformation theory of $\rr$ well-behaves to the theory of algebraic automorphic forms on $GU(n)$. 
The automophy of $\rho$ is reduced to one of $r$ and thus, we seek 
the existence of a corresponding automorphic form on $GU(n)$. 
Once this is established,  
we can use base change and descend techniques back and forth and eventually we find a regular algebraic conjugate self-dual cuspidal automorphic representation of $\GL_n(\A_K)$ corresponding to $\rho$. 

In the works of \cite{CHT} (the minimal case) and \cite{Ta3} (the non-minimal case) with 
a detailed study of the Dwork family in \cite{HSBT}, 
the (potential) automorphy for the symplectic case is studied, i.e., when $n$ is even. 
The results are extended to the orthogonal case by Guerberoff \cite{Gu}.  
After those works, the assumption on Hodge-Tate numbers is relaxed in \cite{Ge1},\cite{Ge2} 
for $\rho$ ordinary at each place dividing $p$ by using Hida family on $GU(n)$. On the other hand, the largeness assumption on the residual image undergoes various changes by \cite{BGHT}, \cite{Th1} for $p>2$ while the non-minimal case is handled in \cite{BGGT} by 
introducing the notation of potentially diagonalizable lifts of $\br|_{G_{F_,v}}$ for 
each place $v$ of $F$ dividing $p$. 
The case when $p=2$ is more delicate and  some of the usual computation involving global or local 
Galois cohomology break down. In \cite{Th2}, the assumption on the residual image is 
replaced with weakly adequateness and Thorne deduced a $2$-adic automorphic lifting theorem 
for the  minimal case. 
Recently, his result is extended in \cite{BCGP2} for the ordinary case and in addition,   
more smaller images can be handled. The results are applied to proving Yoshida conjecture for a positive portion of abelian surfaces over $\Q$.  
We should also remark that automorphy lifting theorems over CM-fields including 
the non-self dual case has been studied in \cite{Allen10} and \cite{MT}. 
In the series of the above important works, the Dowrk family plays very important role in proving 
Khare-Wintenberger type lifting theorems in various settings. In particular, the residual monodormy 
representations which takes the values in $\F_{p^r}$ and local Galois representation theoretic studies are especially important. 
The case when $r=1$ and $p$ is sufficiently large is handled in \cite{HSBT} and later it is extended to any $r\ge 1$ and 
$p$ odd in \cite{BGHT}. In this paper, we study the case $p=2$ for $r$ arbitrary. 
To explain main results we prepare some notation. We refer to the corresponding sections 
for details. 

Let $N$ be an odd positive integer and $n$ be an even positive integer with $N\ge n+1$. 
We denote by $\mu_N$ the group consisting of all $N$-th roots of unity and 
we fix a primitive root $\zeta_N\in \mu_N(\C)$.  
We sometimes view $\mu_N$ as 
a group scheme over $\Z$.  For any subring $R$ of $\C$, we write $R^+:=R\cap \R$ so that 
$\Z[\zeta_N]^+=\Z[\zeta_N+\zeta^{-1}_N]$ and if $\lambda$ is a non-zero prime ideal of $R$, 
we denote by $R_\lambda$ the completion of $R$ at $\lambda$. 
Let $T_0=\mathbb{P}^1\setminus(\{\infty\}\cup \mu_N)
={\rm Spec}\hspace{0.5mm}\Z[\frac{1}{N}][t,\frac{1}{t^N-1}]$.  
We define the Dwork family $X\subset \mathbb{P}^{N-1}\times T_0$ which is given by 
the family of hypersurfaces:
$$
Y=Y_N:X^N_1+\cdots+X^N_N-Nt X_1\cdots X_N=0
$$
defined in $\mathbb{P}^{N-1}_{T_0}$ with $\pi:Y\lra T_0,\ ([X_1:\cdots:X_N],t)\mapsto t$. 
We denote by $Y_t$ the fiber at each point $t$ of $T_0$. 
Let $V_B$ be the locally constant sheaf over $\Z[\frac{1}{N},\zeta_N]^+$ 
defined by (\ref{VB}) which is a submodule of the 
locally constant sheaf $R^{N-1}\pi_\ast \Z[\frac{1}{N},\zeta_N]^+$ which is free of rank $n$ 
over $\Z[\frac{1}{N},\zeta_N]^+$. It inherits a natural perfect alternating pairing 
$\langle \ast, \ast\rangle:V_B\times V_B\lra \Z[\frac{1}{N},\zeta_N]^+$.  
For each $t\in T_0(\C)$, we define the monodromy representation of the (topological) fundamental 
group $\pi_1(T_0(\C),t)$ attached to $V_B$:
$$\rho_{t}:\pi_1(T_0(\C),t)\lra {\rm Sp}(V_{B,t},\langle \ast, \ast\rangle)
\simeq {\rm Sp}_n\Big(\Z\Big[\frac{1}{N},\zeta_N\Big]^+\Big).$$
For each non-zero prime ideal $\la$ of $\Z[\frac{1}{N},\zeta_N]^+$, 
we denote by $\br_{t,{\rm mod}\ \la}$ the reduction of $\rho_t$ modulo $\la$:
$$\br_{t,{\rm mod}\ \la}:\pi_1(T_0(\C),t)\lra  {\rm Sp}_n(k_\la)$$
where $k_\la:=\Z[\frac{1}{N},\zeta_N]^+/\la$. It is known that $\br_{t,{\rm mod}}$ 
is absolutely irreducible (see the proof of Proposition \ref{monod}) and that its image is full \cite[Corollary 4.7, p.59]{BGHT} when 
the residual characteristic is odd. Thus, the case of the even residual characteristic remains to study.   
Then, we prove the following result and we include the odd residual characteristic case for convenience:
\begin{thm}\label{mainMG1}{\rm(}Proposition \ref{monod}{\rm)} Keep the notation in the previous lemma. 
Assume $n\ge 4$. 
Let $\lambda$ be a non-zero prime ideal of $\Z\Big[\frac{1}{N},\zeta_N\Big]^+$ and 
denote by $k_\lambda$ its residue field. Put $l_\lambda=k_\lambda$ if $N>n+1$ and 
$l_\lambda=\F_p$ if $N=n+1$ where $p$ is the rational prime under $\la$.  
\begin{enumerate}
\item If the characteristic of $k_\lambda$ is odd, then $\rho_{t,{\rm mod}\ \lambda}$ 
is surjective onto $\Sp_n(l_{\la})$. 
\item If $\lambda|2$, then the image of $\rho_{t,{\rm mod}\ \lambda}$  is isomorphic to one of the following groups: 
\begin{enumerate}
\item  $\Sp_n(l_\lambda)$;
\item ${\rm O}^{\pm}_n(l_\lambda)$;
\item $S_{n+1}$ if $l_\lambda=\F_2$. This case can happen only when $n=2^m$ 
for some $m\ge 2$; 
\item $S_{n+2}$ if $l_\lambda=\F_2$. This case can happen only when $n=2^m-2$ 
for some $m\ge 3$.
\end{enumerate}
Further, in the cases (c) and (d),  $\rho_{t,{\rm mod}\ \lambda}$ factors through the standard 
representation of $S_{n+1}$ and $S_{n+2}$ respectively.  
\end{enumerate}
\end{thm}
Put 
$$\MD(N,n,k):={\rm Im}(\rho_{t,{\rm mod}\ \lambda})$$ with $k=k_\la$ for some 
non-zero prime ideal $\la$ of  $\Z\Big[\frac{1}{N},\zeta_N\Big]^+$ 
and we simply write it as $\MD_n(k)$ if $N$ and $k=k_\la$ is clear from the context. 
By direct computation we can check the following conjecture for $4\le n\le 120$ 
(see Proposition \ref{concrete2}):
\begin{conj}\label{conj1}Assume $n\ge 4$ is even. 
It holds that 
$${\rm MD}_n(\F_2)={\rm MD}(n+1,n,\F_2)=
\left\{
\begin{array}{cl}
S_{n+1} & (\text{$n=2^m$ for some $m\ge 2$})\\
S_{n+2} & (\text{$n=2^m-2$ for some $m\ge 3$})\\
{\rm O}^+_n(\F_2) & 
(\text{$n\neq 2^m,2^m-2$ for any $m\ge 2$ and $n\equiv 0,6\ {\rm mod}\ 8$}) \\
{\rm O}^-_n(\F_2) & 
(\text{$n\neq 2^m,2^m-2$ for any $m\ge 2$ and $n\equiv 2,4\ {\rm mod}\ 8$}) 
\end{array}\right..
$$
Note that the embeddings defined in Remark \ref{embfromStoO} give 
exceptional isomorphisms 
$$S_5\simeq {\rm O}^-_4(\F_2),\ S_8\simeq {\rm O}^+_6(\F_2).$$
\end{conj}
Thus, each mod 2 monodromy group with the trivial coefficient arising from the Dwork family is away from full. 
We prove ${\rm MD}_{n}(\F_2)=S_{n+1}$ when $n\ge 4$ is a power of 2 in Section \ref{DFII}:
\begin{thm}\label{mainMG2}{\rm(}Theorem \ref{MD2power}{\rm)} Keep the notation as above. If $n\ge 4$ is a power of 2, then 
$${\rm MD}_{n}(\F_2)=S_{n+1}.$$
\end{thm}
For the non-trivial coefficients, however, we see that 
${\rm MD}_{4}(k_\la)=\Sp_4(k_\la)$ if $k_\la\neq \F_2$ (see Proposition \ref{concrete2}). 

As an application, we also study a potentially automorphy lifting theorem and a Khare-Wintenberger type theorem in 
Section \ref{GRII}. 

Monodromy representations are closely related to Galois representations and 
they complement each other's properties. In this vein, we study, in particular, 
the mod 2 Galois representations arising from the Dwork family. 
 
Let $n\ge 2$ be an even integer and put $N=n+1$. 
Let $K$ be a number field and fix an embedding $K\hookrightarrow \bQ$. For each $t\in K$ with $t^{n+1}\neq 1$, we consider 
the smooth affine toric hypersurface $Z_t$ defined by 
$$
  Z_t:    f:=x_1 +\cdots+x_n+ \frac{1}{x_1\cdots x_n} - (n+1)t = 0
 $$
in the $n$-dimensional split torus $\mathbb{T}:=\mathbb{G}^n_m$.
Let $\overline{Z}_t$ be the singular mirror symmetry of the fiber $Y_t$ of the Dwork family for $N=n+1$, which is defined by the closure of $Z_t$ in some projective toric variety (see \cite[Section 6]{Wan}).  
Then, the mirror symmetry $W_t$ of $Y_t$ is given by 
 the smooth crepant resolution of $\overline{Z}_t$. 
The primitive part of $H^{n-1}_{\text{\'et}}(Y_{t,\bQ},\Q(\zeta_N)_\lambda)$  is 
 isomorphic to $H^{n-1}_{\text{\'et}}(W_{t,\bQ},\Q_p)\otimes_{\Q_p}\Q(\zeta_{N})_\lambda$ for 
 any finite place $\lambda$ of $\Q(\zeta_{N})$ with the underlying rational prime $p$ 
By construction, $W_t$ has good reduction at each finite place $v$ of $K$ such that $v\nmid (n+1)$ 
and $\psi^{n+1}-1$ is a $v$-adic unit.  
Put 
$$
V_{t,2}:=H^{n-1}_{\text{\'et}} (W_{t, \overline{\mathbb Q}}, \Q_2)
$$
and  let $\langle \ast,\ast \rangle:V_{t,2}\times V_{t,2}\lra \Q_2(1-n)$ be the 
$G_K$-equivariant alternating perfect pairing defined by the Poincar\'e duality. 

Then, $V_{t,2}$ yields a $2$-adic Galois representation 
$$\rho_{t,2}=\rho_{t,\iota_2}:G_K\lra {\rm GSp}(V_{t,2},\langle \ast,\ast \rangle)\simeq {\rm GSp}_n(\Q_2).$$
Choose a $G_K$-stable lattice $T_{t,2}$ over $\Z_2$ of $V_{t,2}$ so that the above alternating pairing 
preserves the integral structure with respect to $T_{t,2}$. Put $\overline{T}_{t,2}=
T_{t,2}\otimes_{\Z_2}\F_2$. Thus, it yields a mod 2 
Galois representation 
$$
  \br_{t,2}:G_K\lra 
{\rm GSp}(\overline{T}_{t,2},\langle \ast,\ast \rangle_{\F_2})\simeq {\rm GSp}_n(\F_2)
=\Sp_n(\F_2)\subset \SL_n(\F_2)
$$
depending on the choice of $T_{t,2}$.   
We view it as a representation to ${\rm GL}_n(\F_2)$ via the natural inclusion. 
Therefore, one can consider the semisimplification $\br^{{\rm ss}}_{t,2}$ of $\br_{t,2}$. 
As mentioned in Section \ref{basics}, we can choose a symplectic basis so that 
it takes the values in ${\rm GSp}_n(\F_2)$. Hence we have a semisimple mod 2 Galois representation 
$$\br^{{\rm ss}}_{t,2}:G_K\lra  {\rm GSp}_n(\F_2)=\Sp_n(\F_2)$$
associated to $W_t$. 
  
Let us introduce the following trinomial
$$
f_t(x):=n x^{n+1}-(n+1)t x^n+1 \in K[x]
$$
and denote by $K_{f_t}$ its decomposition field over $K$ in $\bQ$. 

Then, we prove the following result which extends the author's previous work with Tsuzuki \cite{TY}:
\begin{thm}\label{mainGalois1}{\rm(}Theorem \ref{image}{\rm)} Assume $n=2^m$ with $m\ge 2$. For $t\in K$ with 
$t^{n+1}\neq 1$ such that the Galois group of $K_{f_t}/K$ is $S_{n+1}$, it holds that 
\begin{enumerate}
\item ${\rm Im}(\br_{t,2})\simeq S_{n+1}$. 
\item Up to conjugacy, $\br_{t,2}$ factors through the standard representation 
$\theta_{n+1}:S_{n+1}\lra \GSp_n(\F_2)$. In particular,  $\br_{t,2}$ is absolutely irreducible and 
$\br_{t,2}\simeq \br^{{\rm ss}}_{t,2}$.  
\end{enumerate}
\end{thm}
To prove this theorem, we use toric geometry and group theory and the proof shows 
the claim may not be true for any $n$ even. 
For example, we prove the following result by using Shioda's theory \cite{Shioda} for 28 bitangent lines and 
the fact that $\MD_6(\F_2)=S_8$. 
We introduce the plane quartic defined by  
\begin{equation}\label{quartic}
C_t:4xy^3 + x^3z - 7txy^2z + 2yz^3=0.
\end{equation}
For $t\in K$, it is smooth if and only if $t^7-1\neq 0$.  
Let $\br_{C_t,2}:G_K\lra {\rm Aut}_{\F_2}({\rm Jac}(C_t)[2](\overline{K}))$ be the 
mod 2 representation associated to $C_t$. 
\begin{thm}\label{mainGalois2}{\rm(}Theorem \ref{image8}{\rm)} For each $t\in K$ with $t^7-1\neq 0$, let $\br_{t,2}:G_K\lra \GSp_6(\F_2)$ be the 
mod 2 Galois representation attached to $\overline{T}_{t,2}$ for $n=6$. Then, it holds that  
\begin{enumerate}
\item $(\br_{t,2})^{{\rm ss}}\simeq (\br_{C_t,2})^{{\rm ss}}$;
 \item There exists a Zariski dense subset $H_K$ of $K$ such that 
 ${\rm Im}(\br_{C_t,2})\simeq S_8$ for any $t\in H_K$ with $t(t^7-1)\neq 0$; 
 \item Suppose ${\rm Im}(\br_{C_t,2})\simeq S_8$. Then, it factors through the 
 standard representation of $S_8$ and thus, $\br_{C_t,2}$ is absolutely irreducible. 
In particular, $\br_{t,2}\simeq \br_{C_t,2}$. 
\end{enumerate}
\end{thm}
 
Finally, we explain a result on automorphy of each fiber of the Dwork quintic family under 
suitable conditions as an application of \cite{TY} and \cite{BCGP2}. 
To save notation and symbols, we only treat the case when $F^+=\Q$. 
We refer to Section \ref{App} for general results.   
Recall the Dwork quintic family 
$$
Y:X^5_1+X^5_2+X^5_3+X^5_4+ X^5_5-5t X_1X_2X_3X_4X_5=0 
\subset \mathbb{P}^4_{{\rm Spec}\hspace{0.5mm}\Q[t,\frac{1}{1-t^5}]}
$$ 
whose each fiber $Y_t$ at $t\in \C$ with $t^5\neq 1$ is a Calabi-Yau threefold defined in 
$\mathbb{P}^4(\C)$. 
For each $t\in \Q\setminus\{1\}$ and each rational prime $\ell$, let us consider 
the primitive part $V_{t,\ell}:=H^3_{\text{\'et}}(Y_{t,\bQ},\Q_\ell)^{{\rm prim}}$ 
of the middle cohomology of $Y_t$.  It has rank 4 and inherits a $G_\Q$-equivariant symplectic perfect pairing. 
Thus, we have the $\ell$-adic Galois representation 
$$\rho_{V_t,\ell}:G_\Q\lra \GSp_4(\Q_\ell)$$
attached to $V_{t,\ell}$.  
Put $f_t(x):=4x^5-5t x^4+1\in \Q[x]$ for $t\in \Q$.  
By \cite[Theorem 1.1]{TY}, 
the mod 2 Galois representation $\br_{V_t,2}$ is absolutely irreducible if $f_t$ is 
irreducible over $\Q$ and thus the 2-adic Galois representation $\rho_{V_t,2}$ is also 
absolutely irreducible as well in this case.  
\begin{thm}\label{thm-app1}
Let $t\in \Q$ with $t<1$. Assume $f_t$ is irreducible over $\Q$ and $\ord_2(t)<0$ or 
$\ord_t(t-1)>0$. 
Then, it holds that 
\begin{enumerate}
\item $\rho_{V_t,2}$ comes from 
a regular algebraic cuspidal representation of $\GSp_4(\A_\Q)$ of weight 0. 
In terms of the classical language, there exists a holomorphic Siegel cusp form $h$ of 
parallel weight $(3,3)$ on $\GSp_4(\A_\Q)$ such that the $\ell$-adic 
representation attached to $\rho_{V_t,\ell}$ is isomorphic to $\rho_{h,\ell}$; 
\item  $\rho^{{\rm ss}}_{V_t,\ell}\simeq \rho^{{\rm ss}}_{h,\ell}$ for each $\ell>2$.
\end{enumerate}
Furthermore, $h$ is a genuine form in the sense of 
\cite[Section 2]{KWY}, hence it is neither CAP, endoscopic, of Asai-type lift, nor of symmetric 
cubic lift unless $t=0$ in which case it is the automorphic induction of a certain Hecke character of 
$\A^\times_{\Q(\zeta_5)}$.  
\end{thm}

\begin{rmk}\label{irrdremark} Let $t\in \Q\setminus\{1\}$ satisfy the assumption in Theorem \ref{thm-app1}.
The Galois representation $\rho_{V_t,\ell}$  is absolutely irreducible for the set of 
rational primes $\ell$ of Dirichlet density one by 
\cite[Theorem 3.2]{CG}.  
\end{rmk}

After a detailed analysis on the non-primitive part, we will prove the following:
\begin{thm}\label{thm-app2}
Let $\ell$ be a rational prime and $t\in \Q\setminus\{1\}$. 
Assume $t<1$, $f_t$ is irreducible over $\Q$ and $\ord_2(t)<0$. Assume further that the polynomial 
$Q_t(x):=x^5+10 x^4+35 x^3+50 x^2+25 x+4 - 4t^{-5}$ in $x$ is irreducible over $\Q$. 
Then, the $L$-function $L(s,H^3_{\text{\'et}}(Y_{t,\bQ},\Q_\ell))$ is extended to 
the whole space of complex numbers as an entire function in $s$.   
\end{thm}
For each number field $K$, there are only finitely many $t\in K\setminus\{0,1\}$ such that 
$Q_t$ is reducible over $K$ (see Remark \ref{irredPt}). We should remark that 
$L(s,H^3_{\text{\'et}}(Y_{t,\bQ},\Q_\ell))$ always has meromorphic continuation 
for each $t\in \Q\setminus\{1\}$ by \cite{BL}.

\begin{rmk}\label{Dworkgenus2} For each $t\in \Q\setminus\{1\}$, consider the hyperelliptic curve 
$D_t$ of genus 2 defined as a projective smooth model of the affine curve  
$y^2+(3+5tx)y=x^5$. 
In Appendix C, we will prove that $\br_{V_t,3}^{{\rm ss}}\simeq \br_{D_t,3}^{{\rm ss}}$ where 
$\br_{D_t,3}:G_\Q\lra \GSp_4(\F_3)$ is the mod 3 Galois representation attached to ${\rm Jac}(D_t)$. 
The hyperelliptic curve $D_t$ of genus 2 has the following properties:
\begin{enumerate}
\item the discriminant of $D_t$ in the sense of \cite[Definition 1.6, p.732]{LH} is 
$3^8\cdot 5^5\cdot (1-t^5)$;
\item there exists a Zariski dense set $H_\Q$ of $\Q\setminus\{1\}$ such that 
${\rm Im}(\br_{D_t,3})=\GSp_4(\F_3)$ for any $t\in H_\Q$; 
\item $D_t$ is not ordinary but potentially nearly ordinary at $3$ when $\ord_3(t)\le 0$; 
\item $D_t$ is semi-stable ordinary but not good ordinary at $2$ when $\ord_2(t)<0$.   
\end{enumerate}
The second property easily follows from a specialization argument and an explicit computation when $t=-1$ so that $\br_{D_{-1},3}$ is surjective. Others are obtained by 
direct computation {\rm(}cf .\cite[Th\'eor\`eme 1, p.204]{Liu}{\rm)}. 
Under the assumption in Theorem \ref{thm-app1}, the  mod 3 Galois representation $\br_{D_t,3}$ comes from a Siegel cusp form $h$ of parallel weight $(3,3)$.  
We hope to find a Siegel cusp form of parallel weight $(2,2)$ giving rise to $\br_{D_t,3}$ from $h$ by extending the theory of Pilloni \cite{Pilloni}.   
\end{rmk}

We organize this paper as follows. 
Section \ref{DFI} and \ref{DFII} are devoted to study the residual monodromy representations 
and also $p$-adic and $\ell$-adic local properties of Galois representations  for 
the Dwork family.  Section \ref{DFI} has a geometric aspect while Section \ref{DFII} has 
various aspects such as group theory and toric geometry. Many things are done already when 
the residual characteristic is odd. Thus, main contribution of this paper is when the residual characteristic is even. 
By using the results in the above sections, after  recalling the basic facts for Galois representations in Section \ref{basics}, 
we prove the potentially automorphy result in Section \ref{GRII} 
to establish a Khare-Wintenberger's type lifting theorem (Theorem \ref{pot-auto}). 
We should remark that the oddness is not important here and we can lift any mod 2 Galois representations who images 
are governed by the residual monodromy groups (see Remark \ref{OFA}). However, oddness may becomes essential in 
even characteristic case for automorphy as in \cite{Th2} and \cite[Section 5]{BCGP2}. 

Several applications of \cite{TY} and \cite{BCGP2} are presented in Section \ref{App}. In particular, we prove the automorphy for many fibers of the Dwork quintic family over totally real fields in a positive proportion. 
We also discuss the non-primitive part and 
prove the strong Hasse-Weil conjecture for the middle \'etale cohomology of each fiber under 
suitable conditions. 
Appendix A is of independent interest and we determine the image of mod 2 Galois representation 
$\br_{t,2}:G_{F^+}\lra \GSp_6(\F_2)$ attached to the Dwork septic family using Shioda's theory of 28-bitangents on smooth plane quartics. 
Appendix B is devoted to computing the cohomology related to 
a mirror model of the Dwork family. Using this, we relate the Dwork quintic 
family with a family of certain hyperelliptic curves of genus 2 in Appendix C as in Remark \ref{Dworkgenus2}.

\textbf{Acknowledgment.} The author would like to thank Professor Henry H. Kim 
for encouragement during the author's stay at the University of Toronto. He is sincerely grateful to Professor Nobuo Tsuzuki for sharing his insights into various aspects of arithmetic and geometry.
The author is also grateful to thank Professors Alan Stapledon, Daqing Wan, Dingxin Zhang, and Masanori Ishida  for discussion on 
smooth crepant resolutions. We would also like to give a special thanks to 
Professor Toby Gee for pointing out some errors in the earlier version and for discussion and encouragement.  

\subsection{Notation}  
We denote by $GSp_{2m}$ the symplectic similitude group defined by 
$J_{2m}:=\begin{pmatrix}
0 & s_m \\
-s_m & 0
\end{pmatrix}
$ where $s_m$ stands for the anti-diagonal matrix of size $m$ whose anti-diagonal entries are all 1. We also denote by $\nu:GSp_{2n}\lra GL_1$ the similitude character.  
For a number field $K$ and  a continuous homomorphism  $r:G_K:={\rm Gal}(\overline{K}/K)\lra \GL_n(\bQ_\ell)$ 
(resp. a continuous homomorphism  $\bar{r}:G_K\lra \GL_n(\bF_\ell)$ and a finite dimensional 
vector space $V$ with a linear continuous $G_K$-action), 
we denote by $r^{{\rm ss}}$ (resp. $\bar{r}^{{\rm ss}}$ and $V^{{\rm ss}}$) the semisimplification of $r$ 
(resp. $\bar{r}$ and $V$). 
By Chebotarev density theorem, the semisimplification is unique up to isomorphism. 
If $n$ is even and ${\rm Im}(r)\subset \GSp_n(\bQ_\ell)$ 
(resp. ${\rm Im}(\bar{r})\subset \GSp_n(\bF_\ell)$, 
then one can find a symplectic basis of the representation space of $r^{{\rm ss}}$ (resp. $\bar{r}^{{\rm ss}}$)  
so that $r^{{\rm ss}}\subset \GSp_n(\bQ_\ell)$ (resp. $\bar{r}^{{\rm ss}}\subset \GSp_n(\bF_\ell)$). 

For two Galois representations $r_1$ and $r_2$ of $G_K$, we write $r_1\simeq r_2$ if they are $G_K$-equivalent 
to each other.  

\section{The Dwork family I}\label{DFI} In this section, we study the residual monodromy representation 
associated to the Dwork family. We treat any residual characteristic including $p=2$. 
We refer to \cite[Section 4 and 5]{BGHT} for basic facts and the case $p>2$ is completed there.  
The results will be used to establish potentially automorphy lifting theorem for 
mod $p$ Galois representations of symplectic type whose coefficient fields 
is $\F_{p^r}$ for an integer $r\ge 1$. The case when $r=1$ and 
$p$ is sufficiently large is handled in \cite{HSBT} and later it is extended to any $r\ge 1$ and 
$p$ odd in \cite{BGHT}. 

\subsection{The Dwork family}\label{DFdef}
Let $N$ be an odd positive integer and $n$ be an even positive integer with $N\ge n+1$. 
We denote by $\mu_N$ the group consisting of all $N$-th roots of unity and 
we fix a primitive root $\zeta_N\in \mu_N(\C)$.  
We sometimes view $\mu_N$ as 
a group scheme over $\Z$.  For any subring $R$ of $\C$, we write $R^+:=R\cap \R$ so that 
$\Z[\zeta_N]^+=\Z[\zeta_N+\zeta^{-1}_N]$ and if $\lambda$ is a non-zero prime ideal of $R$, 
we denote by $R_\lambda$ the completion of $R$ at $\lambda$. 

Let $T_0=\mathbb{P}^1\setminus(\{\infty\}\cup \mu_N)
={\rm Spec}\hspace{0.5mm}\Z[\frac{1}{N}][t,\frac{1}{t^N-1}]$.  
We define the Dwork family $X\subset \mathbb{P}^{N-1}\times T_0$ which is given by 
the family of hypersurfaces:
\begin{equation}\label{dworkeq}
Y=Y_N:X^N_1+\cdots+X^N_N-Nt X_1\cdots X_N=0.
\end{equation}
The natural projection $\pi:Y\lra T_0$ sending $((X_1,\ldots,X_N),t)$ to $t$ is 
smooth of relative dimension $N-2$. For each point $t\in T_0$, we denote by $Y_t$ 
the fiber of $\pi$ at $t$. For each $t\in T_0(\C)$, it is well-known that 
$H^{N-2}(Y_t(\C),\Z)$ is torsion free since $Y_t(\C)$ is a smooth hypersurface in $\mathbb{P}^{N-1}(\C)$. In fact, it follows from well-known arguments 
by using  the Lefschetz hyperplane theorem and the Poincar\'e duality (cf. \cite[Theorem 2.1]{EM}).  

Let $\iota$ be the automorphism of $Y$ defined by the reversion of the coordinates:
$$\iota:((X_1,X_2,\ldots,X_N),t)\mapsto ((X_N,\ldots,X_2,X_1),t).$$ 
Put $H=\mu^N_N/\Delta\mu_N$ where $\Delta\mu_N$ is the image of $\mu_N$ 
under the diagonal embedding $\Delta:\mu_N\hookrightarrow \mu^N_N$. We define 
$H_0=\{\xi=(\xi_1,\ldots,\xi_N)\in H\ |\ 
\xi_1\cdots \xi_N=1\}$ and let $H_0$ act on $Y/\Z[\frac{1}{N},\zeta_N]$ by 
$$\xi (X_1,\ldots,X_N)=(\xi_1 X_1,\ldots,\xi_N X_N).$$ 
For each $\xi=(\xi_1,\ldots,\xi_N)\in H_0$, put 
$\tau(\xi):=\Bigg(\ds\prod_{i=1}^{\frac{N-n-1}{2}}\xi^i_i\Bigg)
\Bigg(\prod_{i=\frac{N+n+3}{2}}^{N}\xi^{i-1}_{i}\Bigg)$ 
and we consider the idempotent 
$$e=\frac{1}{2|H_0|}(\iota+1)\sum_{\xi\in H_0}(\tau(\xi)+\tau(\xi^{-1}))\xi \in 
\Z\Big[\zeta_N,\frac{1}{2N}\Big]^+[H].$$

\subsection{Smooth sheaves}\label{smoothsheaves}

For each non-zero prime ideal $\lambda$ in $\Z[\zeta_N,\frac{1}{N}]^+$ above a rational prime $\ell$, 
we define a $\lambda$-adic smooth sheaf on $(T_0/\Z[\zeta_N,\frac{1}{N\ell}]^+)$ by 
$$V_\lambda=V_{n,\lambda}:=e^\ast R^{N-2}\pi_\ast \Z\Big[\zeta_N,\frac{1}{N}\Big]^+_\lambda$$
which we can view it as a lattice inside $V_\lambda\otimes_{ 
\Z[\zeta_N,\frac{1}{N}]^+_\lambda} {\rm Frac}(\Z[\zeta_N,\frac{1}{N}]^+_\lambda)$ since 
$R^{N-2}\pi_\ast \Z[\zeta_N,\frac{1}{N}]^+_\lambda$ is torsion free. 
Originally it is defined as 
$e^\ast R^{N-2}\pi_\ast \Z[\zeta_N,\frac{1}{2N}]^+_\lambda$ in \cite[p.53]{BGHT} since 
the idempotent $e$ is defined over $ \Z[\zeta_N,\frac{1}{2N}]^+$. However, 
hitting this operation has a bounded denominator at 2 and thus, we can work at the place above 2.  
Since $|H_0|=N^{N-1}$ is a unit in $\Z[\zeta_N,\frac{1}{N}]$, the $H_0$-projection defines 
isomorphism 
\begin{equation}\label{summands}
V_\lambda\otimes_{\Z[\zeta_N,\frac{1}{N}]^+_\lambda}\Z\Big[\zeta_N,\frac{1}{N}\Big]_\lambda 
\stackrel{\sim}{\lra} U_\lambda\otimes_{\Z_\ell} \Big(\frac{1}{2}\Z_\ell \Big)\simeq U_\la
\end{equation}
as a smooth sheaf of $\Z[\zeta_N,\frac{1}{N}]_\lambda$-modules where 
$U_\la$ is the direct summand of the smooth sheave 
$R^{N-2}\pi_\ast \Z[\zeta_N,\frac{1}{N}]_\la$ where $\xi\in H_0$ acts as 
$\tau(\xi)$.  

The cup product induces a perfect alternating pairing $$\langle \ast, \ast\rangle: 
R^{N-2}\pi_\ast \Z\Big[\zeta_N,\frac{1}{N}\Big]^+_\lambda\Big(\frac{N-n-1}{2}\Big)\times 
R^{N-2}\pi_\ast \Z\Big[\zeta_N,\frac{1}{N}\Big]^+_\lambda\Big(\frac{N-n-1}{2}\Big) 
\lra \Z\Big[\zeta_N,\frac{1}{N}\Big]^+_\lambda(1-n).$$
It satisfies $$\langle e^\ast x,e^\ast y \rangle=\langle e_\ast e^\ast x,y \rangle
=\langle e^\ast x, y \rangle$$ since $e^\ast=e_\ast$ and  $e_\ast e^\ast=(e^\ast)^2=e^\ast$. 
Thus, we also have perfect alternating pairings  
$$\langle \ast, \ast\rangle: V_\lambda\Big(\frac{N-n-1}{2}\Big)\times V_\lambda
\Big(\frac{N-n-1}{2}\Big)
\lra \Z\Big[\zeta_N,\frac{1}{N}\Big]^+_\lambda(1-n)$$
and 
$$\langle \ast, \ast\rangle: (V_\lambda/\lambda V_\lambda)\Big(\frac{N-n-1}{2}\Big)\times 
(V_\lambda/\lambda V_\lambda)\Big(\frac{N-n-1}{2}\Big)
\lra \F_\lambda (1-n)$$
where $\F_\lambda$ stands for the residue field of $\lambda$. 
If $\lambda\nmid 2$, then $V_\lambda/\lambda=e^\ast R^{N-2}\pi_\ast (\Z\Big[\zeta_N,\frac{1}{N}\Big]^+/\lambda)=:V[\lambda]$ and thus, for any non-zero ideal $\frak n'=\prod_{i}\lambda^{n_i}_i$ of $\Z\Big[\zeta_N,\frac{1}{N}\Big]^+$ coprime to 2, we have 
$$V[\frak n']:=e^\ast R^{N-2}\pi_\ast \Big(\Z\Big[\zeta_N,\frac{1}{N}\Big]^+/\frak n'\Big)=
\prod_i V_{\lambda_i}/\frak \lambda^{n_i}_i$$
which inherits a perfect alternating pairing as above. 

\begin{lem}Keep the notation being as above. 
The sheaf $V_\lambda$ is locally free of rank $n$ over $\Z\Big[\zeta_N,\frac{1}{N}\Big]^+_\lambda$. 
\end{lem}
\begin{proof}Since 
 $R^{N-2}\pi_\ast \Z\Big[\zeta_N,\frac{1}{N}\Big]^+_\lambda$ is torsion free and finitely 
 generated  over $\Z\Big[\zeta_N,\frac{1}{N}\Big]^+_\lambda$, so is $V_\lambda$. 
The claim follows since 
$V_\lambda \otimes_{\Z[\zeta_N,\frac{1}{N}]^+_\lambda}{\rm Frac}(\Z\Big[\zeta_N,\frac{1}{N}\Big]^+_\lambda$ is locally free of rank $n$ by 
\cite[Lemma 4.1-2.]{BGHT}. 
\end{proof}

The next result follows from \cite[Lemma 5.1]{BGHT}

\begin{lem}\label{pnotl}Let $\lambda$ be a non-zero primes of $\Z\Big[\zeta_N,\frac{1}{N}\Big]^+$. Let $F/\Q(\zeta_N)^+$ be a totally real finite extension and $v$ a 
finite place of $F$ whose residue characteristic is different from one of $\lambda$. 
\begin{enumerate}
\item  If $t\in \O_F$ with $\ord_v(t^N-1)=0$, then the action of $G_F$ on 
$V_{\lambda,t}$ is unramified at $v$.  
\item If $t\in F\setminus\O_F$, then the action of $G_F$ on 
$V_{\lambda,t}$ is tamely ramified at $v$. 
A topological generator of the tame inertia group at $v$ acts via a unipotent matrix  
with the eigenpolynomial $(X-1)^n$. Further, a Frobenius lift has the characteristic polynomial 
$\ds\prod_{i=0}^{n-1}(X-\alpha q_v^i)$ where $q_v=\sharp O_K/v$ and $\alpha\in\{\pm 1\}$. 
\end{enumerate}
\end{lem}

\begin{lem}\label{p=l}Let $\lambda$ be a non-zero prime of $\Z\Big[\zeta_N,\frac{1}{N}\Big]^+$. Let $F/\Q(\zeta_N)^+$ be a totally real finite extension and $v$ a 
finite place of $F$ whose residue characteristic is the same as one of $\lambda$. 
\begin{enumerate}
\item $V_{\lambda,t}|_{G_{F,v}}$ is a de Rham representation of 
Hodge-Tate weights ${\rm HT}_\tau=\{0,1,\ldots,n-1\}$ for any continuous embedding 
$F\hookrightarrow \overline{\Q(\zeta_N)}_\lambda$.   
\item  If $t\in \O_F$ with $\ord_v(t^N-1)=0$, then $V_{\lambda,t}|_{G_{F,v}}$  is crystalline at $v$.  
\item If $t\in F$ with $\ord_v(t)<0$, then $V_{\lambda,t}|_{G_{F,v}}$  is potentially ordinary at $v$.  
\end{enumerate}
\end{lem}
\begin{proof}The first and second claims follow from  \cite[Lemma 5.3-1,2]{BGHT}. 
By potential automorphy of the weakly compatible system $(V_{\lambda',t})_{\lambda'}$ (\cite[Theorem B]{HSBT}), there exists a 
finite totally real extension $M/F$ and a regular algebraic self-dual cuspidal 
automorphic representation $\pi$ of $\GL_n(\A_{M})$ such that $V_{\lambda',t}\simeq 
\rho_{\lambda',\iota}$. Choosing $\lambda'$ whose residue characteristic different from one of 
$\lambda$, then by Lemma \ref{pnotl}-(3) and the local-global compatibility for 
the different residue characteristic case \cite{TaYo}, $\pi_{v}$ is a twisted Steinberg representation. 
Then, the claim follows from \cite[Theorem A]{Ca}.    
\end{proof}

For any scheme $X$ over a totally real field $F$ included in $\bQ$, a point $P\in X(\bQ)$ is said to be totally real if $\iota(P)\in X(\C)$ is fixed by any complex conjugation for any embedding  
$\iota:F\hookrightarrow \C$. We can simply write $P\in X(\R)$ for such a point.  

\begin{lem}\label{infinite}
Let $\lambda$ be a non-zero prime of $\Z\Big[\zeta_N,\frac{1}{N}\Big]^+$. Let $F/\Q(\zeta_N)^+$ be a totally real finite extension
\begin{enumerate}
\item for any $t\in T_0(\bQ)\cap T_0(\R)$ $($(hence, $t$ is a totally real point$)$, 
$V_{\lambda,t}/\la$ is totally odd if $\la\nmid 2$.
\item Assume $\la|2$.  For each infinite place $v$ of $F$, there exists a non-empty open 
subset $\widetilde{\Omega}_v$ of $T_0(F_v)$ including the point $t=0$ such that  
$V_{\lambda,t}/\la$ is strongly residually odd at $v$ for any $t\in \widetilde{\Omega}_v$. 
\end{enumerate}
\end{lem}
\begin{proof}The first claim follows from the argument for \cite[Lemma 3.2]{Ribet} 
with \cite[p.71, line 10-13]{BGHT}) since $\la$ has odd residual characteristic.  

For the second claim, let $F/\Q(\zeta_N)^+$ be a totally real finite extension and 
$\pi_1(T_0)$ be the \'etale fundamental group of $T_0/F$. Then, we have a Galois 
representation $\rho_{\la,T_0}:\pi_1(T_0)\lra \GSp_{n}(\Q(\zeta_N)^+_\la)$.  
An argument for \cite[Section 3.4.1]{Hui} shows there exists a closed point $t\in T_0/F$ 
such that $\rho_{\la,T_0}(\pi_1(T_0))=\rho_{\la,t}(G_{F(t)})$. 
The argument for \cite[Lemma 3.2]{Ribet} shows the action of each complex conjugation is 
non-trivial and thus $t$ has to be totally real.  
When $t=0$, by \cite[Lemma 4.1]{BGHT}, we have $V_{\la,0}\simeq 
\oplus_{i=1}^\frac{n}{2}{\rm Ind}^{G_F}_{G_{F(\zeta_N)}}U_{\frac{N-1}{2}+i,\la}$ as a   
$\Q(\zeta_N)_\la[G_F]$-module (see loc.cit. for the smooth sheave $U_{i,\la}$ of rank one). 
Thus, we can choose a lattice $T_{\la,0}$ of  
$V_{\la,0}\otimes_{\Z[\frac{1}{N},\zeta_N]^+_\la} F_\la$ such that each complex conjugation 
$c_v$ acts as the matrix multiplication of 
$\diag(\overbrace{s,\ldots,s}^{\frac{n}{2}})$ where 
$s=\begin{pmatrix}
0 & 1 \\
1 & 0 
\end{pmatrix}
$. Thus, $T_{\la,0}/\la$ is strongly 
residually totally odd and so is $(V_{\la,0}/\la)^{{\rm ss}}$. 
By \cite[Lemma 2.16-(ii)]{Th2}, we conclude that $V_{\la,0}/\la$ is also strongly 
residually totally odd.   
Let $X$ be the \'etale covering of $T_0/F$ corresponding to the kernel of 
$\br_{\la,T_0}:=(\rho_{\la,T_0}$ mod $\la$). 
Note that $X$ is defined over a totally real extension $F'/F$. 
The results in Section \ref{Mgroups} shows 
$X$ is geometrically connected. Since the fiber at $t=0$ consists of totally real points. 
Thus, by using Euclidean topology, we may take $\widetilde{\Omega}_v$ as 
a sufficiently small open neighborhood at $t=0$ of $T_0(F_v)$ such that 
each fiber of $t\in \widetilde{\Omega}_v$ consists of totally real points. The set 
$\widetilde{\Omega}_v$  is the desired one.      
\end{proof}

\begin{rmk}The Legendre family $E_t:y^2=x(x-1)(x-t)$ over 
$(\mathbb{P}^1\setminus\{0,1,\infty\})/\Q$ yields 
the trivial mod 2 Galois representation at each fiber. 
The image of the monodromy representation of the topological fundamental group 
$\pi_1(\mathbb{P}^1(\C)\setminus\{0,1,\infty\},\ast)$ for this family is $\G(2)=\{\gamma\in 
\SL_2(\Z)\ |\ \gamma\equiv 1_2\ {\rm mod}\ 2\}$ $($cf. \cite[Theorem 1.1.7]{CMSP}$)$ so that 
its mod 2 residual image is trivial as well. On the other hand, we consider  
the Hesse pencil $Y=Y_3:X^3_1+X^3_2+X^3_3-3tX_1X_2X_3=0$ over $(\mathbb{P}^1\setminus\{0,1,\infty\})/\Q$. Then, we can check easily that the special fiber at $t=0$ has a 
mod 2 Galois representation of $G_F$ for any totally real field $F$ such that any complex conjugation acts as 
the matrix multiplication of 
$s=\begin{pmatrix}
0 & 1 \\
1 & 0 
\end{pmatrix}$.  
Thus, to check the oddness for the even residual characteristic case is subtle and we 
need an additional information such as a specialization argument as in Lemma \ref{infinite}.
\end{rmk}

\subsection{Monodromy groups}\label{Mgroups}
Next we consider the Betti side. 
We define the locally constant sheaf 
\begin{equation}\label{VB}
V_B:=e^\ast R^{N-2}\pi_\ast \Z\Big[\frac{1}{N},\zeta_N\Big]^+
\end{equation}
of rank $n$ over $\Z\Big[\frac{1}{N},\zeta_N\Big]^+$. 
As observed in the \'etale case, we have a natural perfect alternating pairing
$$\langle \ast,\ast \rangle:V_B\times V_B\lra \Z\Big[\frac{1}{N},\zeta_N\Big]^+.$$
We introduce $\wT_0=\mathbb{P}^1\setminus\{0,1,\infty\}={\rm Spec}\hspace{0.5mm}
\Z[\frac{1}{N}][\wt,\frac{1}{\wt(1-\wt)}]$ and view $T_0\setminus\{0\}$ as a $\wT_0$-scheme by 
$T_0\setminus\{0\}\lra \wT_0, t\mapsto t^N=\wt$. As in \cite[p.55-56]{BGHT}, we have a twisted model $\widetilde{Y}\subset 
\mathbb{P}^{N-1}\times\wT_0$ of 
$Y$ over $\wT_0$  and the idenpotent $\widetilde{e}$ over $\Z\Big[\frac{1}{2N},\zeta_N\Big]^+$ 
such that 
$$\widetilde{V}_B:=\widetilde{e}^\ast R^{N-2}\widetilde{\pi}_\ast \Z\Big[\frac{1}{N},\zeta_N\Big]^+$$
is the locally constant sheaf of rank $n$ endowed with a natural perfect alternating pairing 
$$\langle \ast,\ast \rangle:\widetilde{V}_B\times \widetilde{V}_B\lra \Z\Big[\frac{1}{N},\zeta_N\Big]^+$$
as observed in the \'etale case. 
The pullback of the pair $(\widetilde{V}_B,\langle \ast,\ast \rangle)$ to 
$T_0(\C)\setminus \{0\}$ to the pair $(V_B,\langle \ast,\ast \rangle)$.

The fundamental group $\pi_1(\wT_0(\C),\wt)$ is generated by 
three elements $\gamma_0,\ \gamma_1$, and $\gamma_\infty$ with the relation 
$\gamma_0\gamma_1\gamma_\infty=1$ where $\gamma_x$ stands for 
a generator of the local monodromy group at $x\in\{0,1,\infty\}$. 

For any fixed $\wt\in \wT_0(\C)$, we have a monodromy representation 
$$\wrho_{\wt}:\pi_1(\wT_0(\C),\wt)\lra \Sp(\widetilde{V}_{B,\wt},\langle \ast,\ast \rangle)\simeq \Sp_{n}
\Big(\Z\Big[\frac{1}{N},\zeta_N\Big]^+\Big).$$
Similarly, for any fixed $t\in T_0(\C)$, we have a monodromy representation 
$$\rho_{t}:\pi_1(T_0(\C),t)\lra \Sp(V_{B,t},\langle \ast,\ast \rangle)\simeq \Sp_{n}
\Big(\Z\Big[\frac{1}{N},\zeta_N\Big]^+\Big).$$
For each non-zero prime ideal $\lambda$ of $\Z\Big[\frac{1}{N},\zeta_N\Big]^+$, 
we write $\wrho_{\wt,{\rm mod}\ \lambda}$ and 
$\rho_{t,{\rm mod}\ \lambda}$ for the reductions modulo $\lambda$ of 
$\wrho_{\wt}$ and $\rho_t$ respectively. 

\begin{lem}\label{red}Keep the notation being as above. For each $t\in T_0(\C)\setminus\{0\}$, 
let $\widetilde{t}=t^N\in\wT_0(\C)$. Then, ${\rm Im}(\wrho_{\wt,{\rm mod}\ \lambda})$ 
is isomorphic to  ${\rm Im}(\rho_{t,{\rm mod}\ \lambda})$ under the isomorphism 
$\Sp(\widetilde{V}_{B,\wt}/\lambda,\langle \ast,\ast \rangle)\stackrel{\sim}{\lra} 
\Sp(V_{B,t}/\lambda,\langle \ast,\ast \rangle)$ induced from the pullback from 
$\wT_0(\C)$ to $T_0(\C)\setminus\{0\}$.
\end{lem}
\begin{proof}We mimic the proof of \cite[Corollary 4.7]{BGHT}.
we have a commutative diagram:
\[
  \xymatrix{
  &  \pi_1(\wT_0(\C),\wt) \ar[r]^{\wrho_{\wt,{\rm mod}\ \lambda}\hspace{5mm}}  & \Sp(\widetilde{V}_{B,\wt}/\lambda,\langle \ast,\ast \rangle) \ar[d]^{\simeq} \\
\pi_1(T_0(\C)\setminus\{0\},t) \ar[r]^{\hspace{3mm}{\rm rest}_\ast} \ar[ru]^{\text{pullback}} & \pi_1(T_0(\C),t) 
\ar[r]^{\rho_{t,{\rm mod}\ \lambda}\hspace{5mm}}   & \Sp(V_{B,t}/\lambda,\langle \ast,\ast \rangle)
  }
\]
where the pushforward of the restriction map is obviously surjective and the cokernel of the pullback is 
cyclic of order $N$. Since $\lambda$ is coprime to $N$, the diagram yields the claim. 
\end{proof}

Thus, the computation of the image of $\rho_{t,{\rm mod}\ \lambda}$ is reduced to one of 
$\wrho_{\wt,{\rm mod}\ \lambda}$. 
Before going further, we introduce two matrices $A,B$ in $\GL_n(\Z[\zeta_N]^+)$ defined as follows. 
Write 
$$(X-1)^n=X_n+A_1X^{n-1}+\cdots+A_n,\ \prod_{j=\frac{N-n+1}{2}}^{j=\frac{N+n-1}{2}}(X-\zeta^j_N)=
X^n+B_1X^{n-1}+\cdots+B_n\in \Z[\zeta_N]^+[X].$$
Note that $A_i=(-1)^i\ds\binom{n}{i}$ ($1\le i \le n$) and $B_n=1$. Then, we define 
\begin{equation}\label{AB}
A=
\begin{pmatrix}
        0 & 0        & \cdots & 0 &-A_n \\
        1 & 0        & \cdots & 0 &-A_{n-1} \\
        0 & 1        & \cdots & 0& -A_{n-2} \\
\vdots & \vdots & \ddots & \vdots  & \vdots \\
0 & 0 & \cdots & 1 & -A_1 \\
\end{pmatrix},\ 
B=
\begin{pmatrix}
        0 & 0        & \cdots & 0 &-B_n \\
        1 & 0        & \cdots & 0 &-B_{n-1} \\
        0 & 1        & \cdots & 0& -B_{n-2} \\
\vdots & \vdots & \ddots & \vdots  & \vdots \\
0 & 0 & \cdots & 1 & -B_1 \\
\end{pmatrix}.
\end{equation}
It is easy to see that there exists $P\in \GL_n(\Z)$ such that 
\begin{equation}\label{fullmonod}
P^{-1}AP=\begin{pmatrix}
        1 & 1     &0   & \cdots & 0  \\
        0 & 1      & \ddots  & \ddots & \vdots  \\
        0 & 0     & \ddots   & \ddots & 0  \\
\vdots & \vdots &\ddots & \ddots & 1   \\
0 & 0 & \cdots& 0 & 1  \\
\end{pmatrix}.
\end{equation}
In fact, we may take $P=(p_{ij})_{1\le i,j\le n}$ with 
$p_{ij}:=
\left\{\begin{array}{cl}
(-1)^{i+j}\binom{n-j}{i-1} & (i+j-1\le n)  \\
0  & (\text{otherwise})
\end{array}\right..
$

For each finite field $k$ of even characteristic and an even positive integer $n=2m$, 
by \cite[p.23, Proposition 1.5.39]{BHRD}, any non-degenerate quadratic map on $k^n$
is equivalent to one of the following quadratic maps:
$$F^+:=\sum_{i=1}^m x_ix_{n+1-i},\ F^-:=\Bigg(\sum_{i=1}^{m-1} x_ix_{n+1-i}\Bigg)+(x^2_m+x_mx_{m+1}+\mu x^2_{m+1})$$
where $\mu\in k^\times\setminus (k^\times)^2$. 
If we regard a quadratic map $F$ on $k^n$ as an algebraic quadratic map over $k$, then 
it is easy to see that $F$ is identically zero on $k^n$ if and only if $F$ is identically zero 
on $\overline{k}$. 
Then, we define 
$${\rm O}^\pm_n(k):=\{g\in \GL_n(k)\ |\ F^\pm(gx)=F^\pm(x),\ \forall x\in k^n\}.$$
Note that ${\rm O}^-_n(k)$ is independent of the choice of $\mu$ up to isomorphism over $k$. 
It follows ${\rm O}^\pm_n(k)\subsetneq \Sp_n(k)$ since $B(x,y):=F^\pm(x+y)-F^\pm(x)-F^\pm(y)=\ds\sum_{i=1}^n x_i y_{n+1-i}$ for $x,y\in k^n$ is 
a non-degenerate symplectic form. The inequality is a consequence of 
\cite[p.32, Theorem 1.6.22]{BHRD}.   

For each positive even integer $n$, let us consider $\F^{n+2}_2$ with 
a perfect alternating pairing $\langle \ast,\ast\rangle:\F^{n+2}_2\times \F^{n+2}_2\lra \F_2,\ (x,y)\mapsto 
\ds\sum_{i=1}^{n+2}x_iy_i$ and we let $S_{n+2}$ act on 
$V$ by permutation of coordinates. Put $W=\Big\{x\in \F^{n+2}_2\ \Big|\ \ds\sum_{i=1}^{n+2}x_i=0\Big\}$ 
and define a line $L=\langle (1,\ldots,1) \rangle\subset W$.  
Put $V:=W/L$ and we denote again by  $\langle\ast,\ast \rangle$ the pairing induced  
from the pairing on $\F^{n+2}_2$. It is easy to see that this pairing is perfect 
alternating. The group $S_{n+2}$ acts on $(V,\langle\ast,\ast \rangle)$ and it yields 
a faithful representation 
\begin{equation}\label{stan}
\theta_{n+2}:S_{n+2}\lra \Sp(\langle\ast,\ast \rangle)\simeq \Sp_n(\F_2)\subset \GL_n(\F_2).
\end{equation}
We regard $S_{n+1}$ as a fixed group of the letter $n+2$ inside $S_{n+2}$ and 
define $\theta_{n+1}:=\theta_{n+2}|_{S_{n+1}}$. 
We call $\theta_{n+2}$ (resp. $\theta_{n+1}$) the standard representation of 
$S_{n+2}$ (resp. $S_{n+1}$). When $n\ge 6$, $\theta_{n+2}$ (resp. $\theta_{n+1}$) is known 
by \cite[Theorem 1.1]{Wagner} as 
an absolutely irreducible faithful representation of $S_{n+2}$ (resp. $S_{n+1}$) 
with the least dimension $n$. We can check the same is true when $n=4$ by applying 
the restriction to $A_6$ or $A_5$ to 
\cite[p.133, Lemma 4.3 and 4.4]{Wagner}.

\begin{prop}\label{monod} Keep the notation in the previous lemma. 
Assume $n\ge 4$. 
Let $\lambda$ be a non-zero prime ideal of $\Z\Big[\frac{1}{N},\zeta_N\Big]^+$ and 
denote by $k_\lambda$ its residue field. Put $l_\lambda=k_\lambda$ if $N>n+1$ and 
$l_\lambda=\F_p$ if $N=n+1$ where $p$ is the rational prime under $\la$.  
\begin{enumerate}
\item If the characteristic of $k_\lambda$ is odd, then $\rho_{t,{\rm mod}\ \lambda}$ 
is surjective onto $\Sp_n(l_{\la})$. 
\item If $\lambda|2$, then the image of $\rho_{t,{\rm mod}\ \lambda}$  is isomorphic to one of the following groups: 
\begin{enumerate}
\item  $\Sp_n(l_\lambda)$;
\item ${\rm O}^{\pm}_n(l_\lambda)$;
\item $S_{n+1}$ if $l_\lambda=\F_2$. This case can happen only when $n=2^m$ 
for some $m\ge 2$; 
\item $S_{n+2}$ if $l_\lambda=\F_2$. This case can happen only when $n=2^m-2$ 
for some $m\ge 3$; 
\end{enumerate}
In the cases of {\rm(}c{\rm)} and {\rm(}d{\rm)}, $\rho_{t,{\rm mod}\ \lambda}$ 
factors through the standard representation $\theta_{n+1}$ or $\theta_{n+2}$ respectively. 
\end{enumerate}
\end{prop}
\begin{proof}We basically follow the proof of \cite[Lemma 4.6]{BGHT}. The first claim is done 
by \cite[Corollary 4.7]{BGHT}. Thus, we assume $\lambda|2$. 
 
Let $\rho:\pi_1(\wT_0(\C),\wt)\lra {\rm SL}_n\Big(\Z\Big[\frac{1}{N},\zeta_N\Big]^+\Big)$ be 
a representation sending $\gamma_0$ to $B^{-1}$ and $\gamma_\infty$ to $A$. 
By the argument for \cite[Proposition 3.3]{BH}, we see that $\br:=\rho$ mod $\lambda$ 
is absolutely irreducible. Since the semisimplifications of $\br$ and 
$\wrho_{\wt,{\rm mod}\ \lambda}$ are isomorphic each other by \cite[Corollary 4.4]{BGHT}, 
$\wrho_{\wt,{\rm mod}\ \lambda}$ is also absolutely irreducible. 

Let $\Delta\subset \pi_1(\wT_0(\C),\wt)$ be the normal subgroup generated by 
all elements which act on $\widetilde{V}_{B,\wt}/\lambda$ as a transvection. 
Notice that $\gamma_1\not\equiv I_n$ modulo $\lambda$ by comparing eigenpolynomials of $A$ 
and $B$. It follows with \cite[Lemma 4.3-2.]{BGHT}  that  $\gamma_1\in \Delta$.  
Then, $\pi_1(\wT_0(\C),\wt)/\Delta$ is a cyclic group generated by 
$\gamma_0\Delta=\gamma_\infty \Delta$. Thus, 
the index $[\wrho_{\wt,{\rm mod}\ \lambda}(\pi_1(\wT_0(\C),\wt)):
\wrho_{\wt,{\rm mod}\ \lambda}(\Delta)]$ divides both $N$ and $2$. Hence,  
$\wrho_{\wt,{\rm mod}\ \lambda}(\pi_1(\wT_0(\C),\wt))=\wrho_{\wt,{\rm mod}\ \lambda}(\Delta)$. 
Thus, ${\rm Im}(\wrho_{\wt,{\rm mod}\ \lambda})$ is an (absolutely) irreducible subgroup 
of $\Sp_n(k_\lambda)$ generated by transvections. 
Further, the argument for the latter part of the proof of \cite[Lemma 4.6]{BGHT} shows 
the image is not contained in $\Sp_n(k)$ for some subfield $k\subsetneq l_\lambda$. 
Applying Kantor's classification (see \cite[Theorem 1.4 and Fig, p.237]{E} for a nice summary), 
${\rm Im}(\wrho_{\wt,{\rm mod}\ \lambda})$ is one of the following up to conjugacy:
\begin{enumerate}
\item $\SL_n(l_\lambda)$,
\item $\SU_n(k)$ when $l_\lambda$ has a subfield $k$ of index 2, 
\item $\Sp_n(l_\lambda)$, 
\item ${\rm O}^\pm _n(l_\lambda)$,
\item $S_{n+1}$ or $S_{n+2}$ when $l_\lambda=\F_2$,
\item $3\cdot {\rm P}\Omega^{-,\pi}_6$ when $n=6$ and $l_\lambda=\F_4$.
\end{enumerate}
We do not recall the last group precisely but it is defined in $\SU_6(\F_2)$ and it contains  
$\SU_4(\F_2)\times \SU_2(\F_2)$ diagonally (see \cite[line 7 from the proof of Proposition 6.3.16, p.350]{BHRD} and the group is denoted by $3^\cdot{\rm U}_4(3)_{\cdot}2_2$ therein). 
Thus, $\SU_4(\F_2)$ is naturally embedded into a subgroup of the last group. 
When $[l_\lambda:\F_2]\ge 3$, applying the argument for \cite[p.59, line 2-4 of the proof of Lemma 4.6]{BGHT}, we can exclude the cases (1), (2) and (6). 
When $[l_\lambda:\F_2]=2$, we can not directly apply the same argument since 
$|l^\times_\lambda\setminus\{1\}|=2$. Instead, if we let $\alpha$ be a generator of $l^\times_\lambda$, then 
the diagonal matrix $g=\diag(\alpha,\alpha,\alpha,\overbrace{1,\ldots,1}^{n-3\ge 1})$ gives an element of 
the groups for (1), (2) and the subgroup $\SU_4(\F_2)$ of the last group. 
However, $g$ is conjugate to an element of $\Sp_n(l_\lambda)$, thus 
the eigenvalues of $g$ and those of $g^{-1}$ should coincide as a multiset. 
Therefore, $\alpha$ has to be 1 and it gives a contradiction. Thus, 
we exclude the cases (1), (2) and (6). We have nothing to do for the case when $l_\lambda=\F_2$. 

Finally, we check the remaining statements for (c) and (d). 
By (\ref{fullmonod}), it is easy to see that $A$ is cyclic of order 
$\ord(A)=\left\{\begin{array}{cl} 
n & (\text{if $n=2^m$ for some $m\ge 2$})  \\
2^{\lfloor\ \log_2 n \rfloor+1} & (\text{otherwise})
\end{array}\right..
$
Thus, $n+1<\ord(A)$ (resp. $n+2<\ord(A)$) unless 
$S_{n+1}$ for $n=2^m$ (resp. $S_{n+2}$ for $n=2^m-2$). 

When $n\ge 6$, the last claim follows from \cite[Theorem 1.1]{Wagner} while  
when $n=4$, the claim follows from \cite[p.133, Lemma 4.3 and 4.4]{Wagner} 
since $\wrho_{\wt,{\rm mod}\ \lambda}$ is absolutely irreducible. 
\end{proof}

\begin{prop}\label{monodn2} Keep the notation in the previous proposition. 
Assume $n=2$. 
Let $\lambda$ be a non-zero prime ideal of $\Z\Big[\frac{1}{N},\zeta_N\Big]^+$ and 
denote by $k_\lambda$ its residue field. Put $l_\lambda=k_\lambda$ if $N>n+1=3$ and 
$l_\lambda=\F_p$ if $N=n+1=3$ where $p$ is the rational prime under $\la$.  
\begin{enumerate}
\item If the characteristic of $k_\lambda$ is odd, then $\rho_{t,{\rm mod}\ \lambda}$ 
is surjective onto $\Sp_2(l_{\la})=\SL_2(l_{\la})$. 
\item If $\lambda|2$, then the image of $\rho_{t,{\rm mod}\ \lambda}$  is isomorphic to 
the dihedral group $D_{2N}$ of order $2N$. 
\end{enumerate}
\end{prop}
\begin{proof}
The first claim is done 
by \cite[Corollary 4.7]{BGHT}. Thus, we assume $\lambda|2$. 
In this case, put $\bar{A}:=A$ mod $\la$ and $\bar{B}:=B$ mod $\la$ respectively where 
$A$ and $B$ are given in (\ref{AB}). 
By direct computation, we can check $\bar{A}\bar{B}\bar{A}^{-1}=\bar{B}^{-1}$ while 
the order of $\bar{A}$ (resp. $\bar{B}$) is $2$ (resp. $N$).  
Thus, we have the claim.
\end{proof}

\begin{rmk}\label{embfromStoO}As mentioned in \cite[line -2 to the bottom in page 350]{Kantor},  
for each even positive integer $n$,  
$S_{n+2}$ is embedded into ${\rm O}_{n+2}(F):=\{g\in \GL_{n+2}(\F_2)\ |\ F(gx)=F(x),\ 
\forall x\in \F^{n+2}_2\}$ 
where $F(x)=\ds\sum_{1\le i\le j\le n+2}x_i x_j$. 
It is easy to see that ${\rm O}_{n+2}(F)$ is isomorphic over $\F_2$ to 
${\rm O}^{\ve}_{n+2}(\F_2)$ where $\ve=-$ if $n\equiv 2$ mod 4 and $\ve=+$ otherwise. 

Let $W=\{x\in \F^{n+2}_2\ |\ \ds\sum_{i=1}^{n+2}x_i=0\}$ and 
$L=\langle (1,\ldots,1) \rangle\subset W$. We write $F_W$ for the restriction of $F$ to $W$.   
If $n\equiv 2$ mod 4, for each $x\in W$, $F_W(x+(1,\ldots,1) )-F_W(x)=\ds\sum_{i=1}^{n+2}x_i=0$.  
It is easy to see that $F_W$ induces a non-degenerate quadratic map $F_V$ on $V:=W/L\simeq \F^n_2$ so that ${\rm O}(F_V)$ is isomorphic over $\F_2$ to ${\rm O}^\pm_n(\F_2)$. 
Thus, we have an embedding from $S_{n+2}$ to ${\rm O}^\pm_n(\F_2)$.

When $n\equiv 0$ mod 4, the above construction breaks down because 
$F_W(x+(1,\ldots,1) )-F_W(x)=1$ if we follows a similar way. 
Instead, we consider 
$W_1:=\{x\in \F^{n+1}_2\ |\ \ds\sum_{i=1}^{n+1}x_i=0\}\simeq \F^n_2$ and let $S_{n+1}$ act on $W_1$ by permutation of coordinates. 
A quadratic map $x\mapsto \ds\sum_{1\le i\le j\le n+1}x_i x_j$ on $\F^{n+1}_2$ induces  
a non-degenerate quadratic map $F_1$  on $W_1$. 
This gives an embedding from $S_{n+1}$ to ${\rm O}^-_n(\F_2)$. 
\end{rmk}

\begin{dfn}\label{MG2}For each finite field $k$ and even positive integer $n\ge 4$, we define 
the Dwork's monodromy group ${\rm MD}_{n}(k)=\MD(N,n,k)$ in $\Sp_n(k)$ by the image of 
$\rho_{t,{\rm mod}\ \lambda}$  for some non-zero prime ideal $\lambda$ of 
$\Z\Big[\frac{1}{N},\zeta_N\Big]^+$  such that $k=k_\lambda$. 
\end{dfn}

Since ${\rm MD}_{n}(k)$ is generated by the matrices $A,B$ mod $\lambda$ inside $\SL_n(k)$ 
for $k=k_\la$, by using Magma \cite{BCP}, we can check the following:
\begin{prop}\label{concrete1}Assume $4\le  n\le 120$ with $n$ even. 
It holds that 
$${\rm MD}_{n}(\F_2)=
\left\{
\begin{array}{cl}
S_{n+1} & (\text{$n=2^m$ for some $m\ge 2$})\\
S_{n+2} & (\text{$n=2^m-2$ for some $m\ge 3$})\\
{\rm O}^+_n(\F_2) & 
(\text{$n\neq 2^m,2^m-2$ for any $m\ge 2$ and $n\equiv 0,6\ {\rm mod}\ 8$}) \\
{\rm O}^-_n(\F_2) & 
(\text{$n\neq 2^m,2^m-2$ for any $m\ge 2$ and $n\equiv 2,4\ {\rm mod}\ 8$}) 
\end{array}\right..
$$
Note that the embeddings defined in Remark \ref{embfromStoO} give 
exceptional isomorphisms 
$$S_5\simeq {\rm O}^-_4(\F_2),\ S_8\simeq {\rm O}^+_6(\F_2).$$
\end{prop}
Thus, each mod 2 monodromy group with the trivial coefficient arising from the Dwork family is away from full. 
We will prove ${\rm MD}_{n}(\F_2)=S_{n+1}$ when $n\ge 4$ is a power of 2 in the next section.
However, for the non-trivial coefficients, as an example, we have:  
\begin{prop}\label{concrete2}Let $n=4$ and $N>5$ be an odd integer. 
Assume $\la$ is a non-zero prime in $\Z[\zeta_N]^+$ above 2. Then, it holds that ${\rm MD}_{4}(k_\la)=\Sp_4(k_\la)$. 
\end{prop}
\begin{proof} 
First we note that $[k_\la:\F_2]\ge 2$ since $N>5$. 
Put $z:=\zeta_N$ mod $\lambda$. 
By definition,     
$$A=
\begin{pmatrix}
0 & 0 &0 & 1 \\
1 & 0 &0 & 0 \\
0 & 1 &0 & 0 \\
0 & 0 &1 & 0 
\end{pmatrix},\ 
B=
\begin{pmatrix}
0 & 0 &0 & 1 \\
1 & 0 &0 & B_3 \\
0 & 1 &0 & B_2 \\
0 & 0 &1 & B_1 
\end{pmatrix}\in \Sp_4(k_\la)
$$
where $B_1=B_3=z^{\frac{N-3}{2}}+z^{\frac{N-1}{2}}+z^{\frac{N+1}{2}}+
z^{\frac{N+3}{2}}$ and $B_2=z+z^2+z^{-2}+z^{-1}$. It is easy to see that 
$B_1,B_2$ are non-zero. 
By Proposition \ref{monod}, we may check there are no non-trivial quadratic maps on 
$k^4_\la$ which are preserved by $A$ and $B$. 
Assume such a quadratic map 
$F:k^4_\la\lra k_\la$ exists. By the invariance for $A$, the quadratic map $F$ takes the following form 
$$F(x_1,x_2,x_3,x_4):=a_1 (x_1^2 + x_2^2 + x_3^2 + x_4^2) + a_2 (x_1 x_2 + x_2 x_3 + x_3 x_4 + 
x_4 x_1) +  a_3 (x_1 x_3 + x_2 x_4),\ a_1,a_2,a_3\in k_\la$$
By the invariance for $B$, we have 
$$a_1B_2+a_3=0,\ a_2+a_1 B_1=0.$$
If $B_2=1$, then $z^5=1$ which contradicts the assumption $N>5$. 
Hence, $B_2\neq 1$. Thus, $a_1=a_3=0$ and it yields $a_2=0$. 
Therefore, $F$ is identically zero.  
\end{proof}

\subsection{Connected-ness}\label{connected}
We will prove a key result for proving a lifting theorem of 
Khare-Wintenberger type so that we can switch the residual characteristic from 
$2$ to another odd prime.  

Let $\frak n=\lambda \frak n'$ be a non-zero squarefree ideal of $\Z\Big[\frac{1}{N},\zeta_N\Big]^+$ such that $\lambda$ is a prime ideal above $2$ and $\frak n'$ is coprime to 2. 
Further, we assume no two distinct prime factors of $\frak n'$ have the same residual characteristic.   
Let $F$ be a finite extension of $\Q(\zeta_N)^+$. 
Suppose $W$ is a finite free $\Z\Big[\frac{1}{N},\zeta_N\Big]^+/\frak n$-module of rank $n$ 
with a continuous action of $G_F={\rm Gal}(\overline{F}/F)$ and a perfect alternating pairing 
$$\langle \ast,\ast \rangle_W:W\times W\lra 
(\Z\Big[\frac{1}{N},\zeta_N\Big]^+/\frak n)(1-n)$$
such that $\langle gw_1,gw_2 \rangle=g\langle w_1,w_2 \rangle$ for 
$g\in G$ and $w_1,w_2\in W$. We denote by 
$\br_W:G_F\lra \GSp(W,\langle\ast,\ast\rangle_W)$ the corresponding 
Galois representation. 
We consider the functor from 
the category of $T_0\times_{\Z[\frac{1}{N}]}F$-schemes to the category of sets 
which sends $S$ to 
$$\{\phi:W_S\stackrel{\sim}{\lra}(V_\lambda/\lambda\times 
V[\frak n'])(\frac{N-n+1}{2})_S\ |\ \text{$\phi$ preserves the symplectic structures} \}.$$
Here $W_S$ is the smooth obtained as the pullback of $W$ under 
the structure morphism $S\lra {\rm Spec}\hspace{0.5mm}F$ and  
$\phi$ is isomorphic as smooth sheaves over $S_{\text{\'et}}$. 
This functor is represented by a finite etale cover $\pi_W:T_W\lra T_0\times_{\Z[\frac{1}{N}]}F$. 

\begin{prop}\label{conn}Keep the notation being as above. Let $n\ge 2$ be an even integer. 
Assume $\br_W$ mod $\lambda$ is conjugate to a subgroup of ${\rm MD}_n(k)_\lambda$. 
Then, $T_W$ is geometrically connected. 
\end{prop}
\begin{proof}Pick a point $x=\phi_x\in T_W(\C)$ with the underlying point $t=\pi_W(x)\in T_0(\C)$. 
By using a comparison isomorphism between etale and singular cohomologies, 
we have a monodromy representation  
$$\pi_1(T_0(\C),t)\lra \Sp(W_t)\stackrel{\sim}{\lra} \Sp((V_\lambda/\lambda\times 
V[\frak n'])_t),\ \gamma \mapsto \phi_{\widetilde{\gamma}_x(1)}\circ \phi^{-1}_x$$
where the path $\widetilde{\gamma}_x:[0,1]\lra T_W(\C)$ stands for a lift of $\gamma$ such that 
$\widetilde{\gamma}_x(0)=x$. 
The projection of the monodromy representation to 
$\Sp(V[\frak n'])$ is surjective by \cite[Corollary 4.8]{BGHT} while 
the image to $\Sp(V_\lambda/\lambda)$ is ${\rm MD}_n(k_\lambda)$ by definition. 
Now pick another point $x'=\phi_{x'}$ in $T_W(\C)$ and put $g:=\phi_{x'}\phi^{-1}_x$. 
Then, $g$ mod $\lambda$ belongs to ${\rm MD}_n(k_\lambda)$ by assumption while 
obviously $g$ mod $\frak n'$ belongs to $\Sp(V[\frak n'])$. 
By Goursat's lemma (\cite[p.793, Lemma 5.2.1]{Ribet}), $g$ comes from 
the image of $\pi_1(T_0(\C),t)$. Hence, there exists an element $\gamma\in \pi_1(T_0(\C),t)$ 
such that $g=\phi_{\widetilde{\gamma}_x(1)}\circ \phi^{-1}_x$. 
Thus, $x'=\phi_{x'}=\phi_{\widetilde{\gamma}_x(1)}=\widetilde{\gamma}_x(1)$ which means that 
$x=\widetilde{\gamma}_x(0)$ is connected to $x'$ by a path inside $T_W(\C)$.  This completes the proof.
\end{proof}

\section{The Dwork family II:}\label{DFII}
In this section, we will study the mod 2 image of the Galois representation 
attached to each fiber defined over a number field of the Dwork family when $N=n+1$ and $n\ge 4$ even. 
The case $n=4$ is done in the author's previous work with Tsuzuki \cite{TY} and 
we follow some of strategies established therein.

We recall that $V_\lambda$ defined in the previous section has large coefficient in general. 
However, when $N=n+1$, we can work on a mirror model (cf. \cite{Batyrev}, \cite{Wan}) of the Dwork family to avoid the large coefficient. 

Let $K$ be a number field and fix an embedding $K\hookrightarrow \bQ$. For each $t\in K$ with $t^{n+1}\neq 1$, we consider 
the smooth affine toric hypersurface $Z_t$ defined by 
\begin{equation}\label{Zt}
  Z_t:    f:=x_1 +\cdots+x_n+ \frac{1}{x_1\cdots x_n} - (n+1)t = 0
\end{equation} 
in the $n$-dimensional split torus $\mathbb{T}:=\mathbb{G}^n_m$.
Let $\overline{Z}_t$ be the singular mirror symmetry of the fiber $Y_t$ of the Dwork family (\ref{dworkeq}), which is defined by the closure of $Z_t$ in the projective toric variety $\mathbb P_\Delta$ (see \cite[Section 6]{Wan}).  
Here $\Delta$ is the Newton polytope of $f$ inside $\R^n$ with vertices $\{e_1,\ldots,e_n,-(e_1,\ldots,-e_n)\}$. Here $\{e_i\}_{i=1}^n$ is the standard basis of $N_\R:=\R^n$. Then we have a decomposition 
$\mathbb P_\Delta=\coprod_{\tau\prec \Delta}\mathbb{T}_{\Delta,\tau}$ where $\tau$ runs over all faces of $\Delta$ 
(\cite[Section 6]{Wan} and \cite[p.497]{Batyrev}).
Let $\Delta^\ast$ be the dual polytope in $N^\ast_\R={\rm Hom}_\R(\R^n,\R)$  
with vertices $\{(n+1)e^\ast_i-(e^\ast_1+\cdots+e^\ast_n)\ |\ i=1,\ldots,n\}\cup\{-(e^\ast_1+\cdots+e^\ast_n)\}$ 
where $\{e^\ast_i\}_{i=1}^n$ stands for the dual basis of $\{e_i\}_{i=1}^n$.  
We fix a regular triangulation $\mathcal T$ of $\Delta^\ast$ (see  \cite[Appendix]{LZ} 
or \cite[the proof of Theorem 7.5 with the star decomposition]{Wan} 
for the existence). 
Then, the mirror symmetry $W_t$ of $Y_t$ is given by 
 the smooth crepant resolution of $\overline{Z}_t$ associated to the triangulation $\mathcal T$  
 (see \cite[the proof of Theorem 2.2.24, first paragraph, Proposition 2.2.22(ii)]{Batyrev}).  
The primitive part of $H^{n-1}_{\text{\'et}}(Y_{t,\bQ},\Q(\zeta_N)_\lambda)$  is 
 isomorphic to $H^{n-1}_{\text{\'et}}(W_{t,\bQ},\Q_p)\otimes_{\Q_p}\Q(\zeta_{N})_\lambda$ for 
 any finite place $\lambda$ of $\Q(\zeta_{n+1})$ with the underlying rational prime $p$ 
 (this follows from the construction \cite[Section 5.1]{Batyrev} with the 
 argument in \cite[Section 7]{Wan}).  
By construction, it is easy to see that $W_t$ has good reduction at each finite place $v$ of $K$ such that $v\nmid (n+1)$ 
and $\psi^{n+1}-1$ is a $v$-adic unit.  
Put 
\begin{equation}\label{vt2}
V_{t,2}:=H^{n-1}_{\text{\'et}} (W_{t, \overline{\mathbb Q}}, \Q_2)
\end{equation}
and  let $\langle \ast,\ast \rangle:V_{t,2}\times V_{t,2}\lra \Q_2(1-n)$ be the 
$G_K$-equivariant alternating perfect pairing defined by the Poincar\'e duality. 

Then, $V_{t,2}$ yields a $2$-adic Galois representation 
$$\rho_{t,2}=\rho_{t,\iota_2}:G_K\lra {\rm GSp}(V_{t,2},\langle \ast,\ast \rangle)\simeq {\rm GSp}_n(\Q_2).$$
Choose a $G_K$-stable lattice $T_{t,2}$ over $\Z_2$ of $V_{t,2}$ so that the above alternating pairing 
preserves the integral structure with respect to $T_{t,2}$. Put $\overline{T}_{t,2}=
T_{t,2}\otimes_{\Z_2}\F_2$. Thus, it yields a mod 2 
Galois representation 
\begin{equation}\label{mod2}
  \br_{t,2}:G_K\lra 
{\rm GSp}(\overline{T}_{t,2},\langle \ast,\ast \rangle_{\F_2})\simeq {\rm GSp}_n(\F_2)
=\Sp_n(\F_2)\subset \SL_n(\F_2)
\end{equation} 
depending on the choice of $T_{t,2}$.   
We view it as a representation to ${\rm GL}_n(\F_2)$ via the natural inclusion. 
Therefore, one can consider the semisimplification $\br^{{\rm ss}}_{t,2}$ of $\br_{t,2}$. 
As mentioned in Section \ref{basics}, we can choose a symplectic basis so that 
it takes the values in ${\rm GSp}_n(\F_2)$. Hence we have a semisimple mod 2 Galois representation 
$$\br^{{\rm ss}}_{t,2}:G_K\lra  {\rm GSp}_n(\F_2)$$
associated to $W_t$. 
  
Let us introduce the following trinomial
\begin{equation}\label{tri}
f_t(x):=n x^{n+1}-(n+1)t x^n+1 \in K[x]
\end{equation} 
and denote by $K_{f_t}$ its decomposition field over $K$ in $\bQ$. 
Notice that the discriminant of $f_t$ is given by $(-1)^{\frac{n}{2}}n^n(n+1)^{n+1}(1-t^{n+1})$ 
(cf. \cite[p.1105, Theorem 2]{Swan}). 
In this section, we prove the following result:
\begin{thm}\label{image}Assume $n=2^m$ with $m\ge 2$. For $t\in K$ with 
$t^{n+1}\neq 1$ such that the Galois group of $K_{f_t}/K$ is $S_{n+1}$, it holds that 
\begin{enumerate}
\item ${\rm Im}(\br_{t,2})\simeq S_{n+1}$. 
\item Up to conjugacy, $\br_{t,2}$ factors through the standard representation 
$\theta_{n+1}:S_{n+1}\lra \GSp_n(\F_2)$. In particular,  $\br_{t,2}$ is absolutely irreducible and 
$\br_{t,2}\simeq \br^{{\rm ss}}_{t,2}$.  
\end{enumerate}
\end{thm}
As an application, we have 
\begin{thm}\label{MD2power} If $n\ge 4$ is a power of 2. 
Then, $\MD_n(\F_2)$ is isomorphic to $S_{n+1}$ and the monodromy representation for 
$\MD_n(\F_2)$ factors through  the standard representation 
$\theta_{n+1}$ of $S_{n+1}$. 
\end{thm}
\begin{proof}We view $t$ as a variable and $f:=f_t$ as a polynomial over $\Q(t)$. 
It is easy to see that $f_t$ is irreducible over $\bQ(t)$ because its roots are 
interpreted as a fiber of a finite morphism $\mathbb{P}^1\lra \mathbb{P}^1,\ x\mapsto \frac{n x^{n+1}+1}{(n+1)x^n}$. Let $M_f$ be the decomposition field of $f$ over $\Q(t)$. 
By \cite[p.175, Corollary 10.2.2-(b)]{FJ}, $M_f/\Q(t)$ is regular 
(see \cite[p.39, Section 2.6]{FJ} for regularity extensions). 
By \cite[Proposition 3.3.6]{JLY}, for any finite extension $K/\Q$, 
$KM_f/K(t)$ is regular. Thus, for each $K$, the Hilbert set $H_K(f_t)\subset K$ does exists 
so that 
$${\rm Gal}(K_{f_t}/K)\simeq {\rm Gal}(KM_f/K(t))\simeq {\rm Gal}(M_f/\Q(t))\simeq S_{n+1}$$
for any $t\in H_K(f_t)$. 
The last isomorphism can be checked as follows. Put $g_t(x)=x^{n+1}f_t(\frac{1}{x})=
x^{n+1}-(n+1)tx+n$. Then, it is easy to see the superelliptic curve 
$-ny^{n+1}=x^{n+1}-(n+1)tx$ is smooth over $\Q(t)$ and hence, 
geometrically irreducible over $\Q(t)$. By 
\cite[p.175, Corollary 10.2.2-(b)]{FJ}, the splitting field 
of $x^{n+1}-(n+1)tx-ny^{n+1}$ over $\Q(t,y)$ is regular. Hence, 
its Galois group over $\Q(t,y)$ is isomorphic to one of 
$h(x,y,t):=x^{n+1}-(n+1)tx+-ny$ over $\Q(t,y)$. By \cite[Theorem 1]{Uchida}, 
the Galois group of $h(x,y,t)$ over $\Q(t,y)$ is isomorphic to $S_{n_+1}$. 
Summing up everything, there exists the Hilbert set $H_\Q(h)\subset \Q^2$ such that 
for each $(t,y)\in H_\Q(h_t)$, 
the splitting field of $x^{n+1}-(n+1)tx-ny^{n+1}$ over $\Q$ has $S_{n+1}$ 
as its Galois group. Replacing $(\frac{y}{x},y^n t)$ with $(x,t)$, we conclude 
${\rm Gal}(M_f/\Q(t))\simeq S_{n+1}$. 

For $N=n+1$ and a non-zero prime ideal $\la$ of $\Z[\zeta_N]^+$ above 2, 
we recall the residual monodromy representation
$$\rho_{t,{\rm mod}\ \la}:\pi_1(T_0(\C),t)\lra \Sp_n(l_\la)=\Sp_n(\F_2)\subset \Sp_n(k_\la)$$
which is equivalent to the mod 2 monodromy representation associated to 
the mirror family $\{W_t\}_t$. 
Since $T_0$ is defined over $F$ and it is geometrically connected, 
the profinite completion of $\pi_1(T_0(\C),t)$ is isomorphic to 
${\rm Gal}(\overline{\C(t)}/\C(t))\simeq {\rm Gal}(\overline{\bQ(t)}/\bQ(t))$. 
Thus, $\rho_{t,{\rm mod}\ \la}$ factors through 
the Galois group of some finite extension of $\bQ(t)$. Then, it descends to 
$\br:G_{K(t)}\lra \Sp_n(\F_2)$ for any sufficiently large finite extension $K/\Q(\zeta_{n+1})^+$ such that  $\overline{K(t)}^{{\rm Ker}(\br)}$ is a regular  extension over $K(t)$. 
For each $\alpha \in \bQ$, we denote by $D_\alpha\simeq {\rm Gal}(\bQ/K(\alpha))$ 
the decomposition group at $\alpha$. By construction, we have 
$\br|_{D_\alpha}\simeq \br_{\alpha,2}|_{G_{K(\alpha)}}$. 
By Theorem \ref{image} and the above argument, there exists a Zariski dense subset $U\subset K$ such that ${\rm Im}
(\br_{\alpha,2}|_{G_{K(\alpha)}})\simeq S_{n+1}$ for any $\alpha\in U$. 
Increasing $K$, the claim follows from the higher dimensional Chebotarev density theorem \cite[Section 9.3]{Serre}.
\end{proof}

\subsection{Reciprocity}\label{Reciprocity} 
Let $K/\Q$ be a finite extension.  
Let $t\in K$ with $t^{n+1}\neq 1$. 
For each finite field $k$, we define  
$$n(f_t,k):=\sharp\{x\in k\ |\ f_t(x)=0\}.$$
We will relate $n(f_t,k)$ with the mod 2 representation $\br_{t,2}$. 

\begin{lem}\label{ZZ}
Assume $n\ge 4$ is even. For any finite extension $M/K$ and a finite place $v$ of $M$ such that 
$v\nmid 2(n+1)$ and $\ord_v(t^{n+1}-1)=0$, it holds that 
$$\sharp \overline{Z}_t(k_v)\equiv \sharp Z_t(k_v)+1\ {\rm mod}\ 2.$$
\end{lem}
\begin{proof}The claim follows the formula (11) in \cite[p.171]{Wan} (note that $Y_\lambda$ and $N_g$ 
therein should $\bar{Z}_t$ and $Z_t$ respectively with $t=\la$).
\end{proof}

\begin{lem}\label{Zn}Assume $n\ge 4$ is a power of 2. For any finite extension $M/K$ and a finite place $v$ of $M$ such that 
$v\nmid 2(n+1)$ and $\ord_v(t^{n+1}-1)=0$, it holds that 
$$\sharp \overline{Z}_t(k_v)\equiv n(f_t,k)+1\ {\rm mod}\ 2.$$ 
\end{lem}
\begin{proof}By Lemma \ref{ZZ}, we have 
$$\sharp \overline{Z}_t(k_v)\equiv \sharp Z_t(k_v)+1\ {\rm mod}\ 2.$$
Let $\tau=(12\cdots n)\in S_n$ acts on $Z_t(k_v)$ by permutation of the coordinates. 
Since $n$ is a power of 2, the number of the $\tau$-orbit of a point 
$(x_1,\ldots,x_n)\in Z_t(k_v)$ is even unless $x:=x_1=\cdots =x_n$. 
Substituting $x:=x_1=\cdots =x_n$ into (\ref{Zt}), we have the claim. 
\end{proof}
The following result also follows from Theorem \ref{pointCount} but
we give another proof using ``mod 2'' feature. 
\begin{lem}\label{YW} Keep the notation and assumption  in the previous lemma.  
Then, $$\sharp W_t(k_v)\equiv \sharp\overline{Z}_t(k_v)\ {\rm mod}\ 2.$$ 
\end{lem}
\begin{proof}Put $q_v:=\sharp k_v$. 
Let $\Phi_{\mathcal T}:W_{t}\lra \overline{Z}_t$ be the smooth crepant resolution associated to $\mathcal{T}$. 
The projective toric hypersurface $\overline{Z}_t$ has the following decomposition 
$\overline{Z}_t=\coprod_{\tau\prec \Delta}Z_{t,\tau},\ Z_{t,\tau}=\overline{Z}_{t}\cap \mathbb{T}_{\Delta,\tau}$ 
(see \cite[p.505]{Batyrev}). 
By the argument for the proof of Proposition 3.2.1 and p.516-517 for a convenient stratification in \cite{Batyrev}, 
$\Phi_{\mathcal T}^{-1}(Z_{t,\tau})$ is stratified by the products $Z_{t,\tau}\times \mathbb{G}^k_m$ 
for some non-negative integer $k$ by using the faces $\tau'$ of $\tau$ and the case $k=0$ corresponds to $\tau'=\tau$. 
Since $\sharp\mathbb{G}^k_m(q_v)=(q_v-1)^k\equiv 0$ mod 2 unless $k=0$.   
Thus, we have $\sharp \Phi_{\mathcal T}^{-1}(Z_{t,\tau})(k_v)\equiv \sharp Z_{t,\tau}(k_v)$ mod 2. The claim follows.  
\end{proof}

\begin{lem}\label{reci} Keep the notation and assumption  in the previous lemma.  
It folds that 
$$\tr(\br_{t,2}({\rm Frob}_v))=n(f_t,k_v)+1\ {\rm mod}\ 2.$$
\end{lem}
\begin{proof}
Put $q_v:=\sharp k_v$ and 
$a_{i,v}:=\tr({\rm Frob}_v|H^i_{\text{\'et}}(W_{t,\overline{K}},\Q_2))\in \Z$ for $0\le i\le 2(n-1)$. 
Since $n$ is even, by Poincar\'e duality, we have 
$$\sharp W_t(k_v)=a_{n-1,v}+\sum_{i=0}^{n-2}(-1)^i (1+q^{n-1-i}_v)a_{i,v}.$$
Since $q_v\equiv 1$ mod 2, 
$\sharp W_t(k_v)\ {\rm mod}\ 2 =a_{n-1,v} \ {\rm mod}\ 2=\tr(\br_{t,2}({\rm Frob}_v))$. 
Thus, the claim follows from Lemma \ref{Zn} and Lemma \ref{YW}. 
\end{proof}

\subsection{A group theoretic argument}
Let $K_{\br_{t,2}}=\bQ^{{\rm Ker}(\br_{t,2})}$. 
\begin{lem}\label{order}Keep the assumption in Theorem \ref{image}.  
\begin{enumerate}
\item If an element $\sigma\in {\rm Gal}(K_{\br_{t,2}}K_{f_t}/K)$ maps to 
an element in ${\rm Gal}(K_{f_t}/K)\simeq S_{n+1}$ of order $n+1$ by restriction, 
then $\sigma$ is of order $n+1$.  
\item Any element of  $\sigma\in {\rm Gal}(K_{\br_{t,2}}K_{f_t}/K_{\br_{t,2}})$ 
maps to an element in ${\rm Gal}(K_{f_t}/K_{f_t}\cap K_{\br_{t,2}})\subset {\rm Gal}(K_{f_t}/K)$ 
by restriction, then $\sigma$ is of order strictly less than $n+1$. 
\end{enumerate} 
\end{lem} 
\begin{proof}We prove the first claim. For each $i\ge 1$, 
put $\det(I_n-x\br_{t,2}(\sigma^i))=a^{(i)}_{0}+a^{(i)}_1 t+\cdots+a^{(i)}_n x^n\in \F_2[x]$ with 
$a^{(i)}_0=a^{(i)}_n=1$. 
For each $i\ge 1$, by Chebotarev density theorem, there exists a place 
$v\nmid 2$ of $K$ such that 
$\br_{t,2}$ is unramified at $v$ and $\sigma^i={\rm Frob}_v$. 
By assumption, $f_t$ is irreducible over $k_v$ since $\deg(f_t)=n+1$ and $\sigma|_{K_{f_t}}$ 
has order $n+1$.  Thus, $n(f_t,k_v)=0$.  
Similarly, for each $i\ge 1$, by Chebotarev density theorem, 
there exists a place $w\nmid 2$ of $K$ such that 
$\br_{t,2}$ is unramified at $w$ and $\sigma^i={\rm Frob}_w$. 
Since $\sigma^i|_{K_{f_t}}$ is order of less than or equal to $n+1$, $n(f_t,k_w)=0$.   
It follows from this and Lemma \ref{YW} that $a^{(i)}_1=1$ for any $i\ge 1$.     
Put $S_i:=a^{(i)}_1$ and $a_j:=a^{(1)}_j$ ($1\le j\le n$) for simplicity. 
By using Newton's formula 
$$S_\nu+a_1 S_{\nu-1}+a_2 S_{\nu-2}+\cdots+a_{\nu-1}S_1+\nu a_v=0\ 
(1\le \nu\le 2n-1)$$
inductively, we see that $a_{2\nu}+a_{2\nu+1}=0$ ($1\le \nu \le \frac{n}{2}-1$) and 
$a_{n-2\nu-1}+a_{n-2\nu}=0$ as well by Poincar\'e duality.  
Thus, we must have $a_i=1$ ($2\le i\le n-1$). Hence 
$\det(I_n-x\br_{t,2}(\sigma))=\frac{1-x^{n+1}}{1-x}$ and $\sigma$ has order $n+1$. 

For the second claim, by a similar argument we can choose a finite place $w$ of 
$K_{\br_{t,2}}$ such that $\sigma={\rm Frob}_w$. Obviously $\br_{t,2}(\sigma)=I_n$ and 
by Lemma \ref{YW}, $n(f_t,k_w)\equiv 1$ mod 2 and in particular, $n(f_t,k_w)\neq 0$. Thus, $f_t$ has to be reducible over $k_w$.
The claim follows from this. 
\end{proof}

\subsection{A proof}We are now ready to prove Theorem \ref{image}. 
For the first claim, let us consider a natural surjection 
$${\rm Gal}(K_{\br_{t,2}}K_{f_t}/K_{\br_{t,2}})\lra 
{\rm Gal}(K_{f_t}/K_{f_t}\cap K_{\br_{t,2}}).$$
By Lemma \ref{order}-(2), ${\rm Gal}(K_{f_t}/K_{f_t}\cap K_{\br_{t,2}})$ 
has no elements of order $n+1$. On the other hand, it is a normal subgroup of 
${\rm Gal}(K_{f_t}/K)\simeq S_{n+1}$. Thus, ${\rm Gal}(K_{f_t}/K_{f_t}\cap K_{\br_{t,2}})$ 
has to be trivial and we have $K_{f_t}\subset K_{\br_{t,2}}$.

Henceforth we assume $n\ge 6$ since the case $n=4$ is done in \cite[Theorem 1.1]{TY}.
Put $G={\rm Gal}(K_{\br_{t,2}}/K)$, $N={\rm Gal}(K_{\br_{t,2}}/K_{f_t})$, and 
$H={\rm Gal}(K_{f_t}/K)$ so that $G/N\simeq H\simeq S_{n+1}$. 
We view $G$ as a subgroup of $\Sp_n(\F_2)$ via $\br_{t,2}$. 
For any rational odd prime $\ell$, we denote by $N_\ell$ a Sylow $\ell$-subgroup of 
$N$. Since $\ell\neq 2$, $N_\ell\subset \Sp_n(\F_2)\subset \SL_n(\bF_2)$ and $N_\ell$ is solvable, $N_\ell$ is embedded into a subgroup of $T(\F_{2^m})\simeq (\F^\times_{2^m})^{n-1}$ some 
$m\ge 1$ where $T$ is the diagonal torus of $SL_n$ 
(cf. \cite[p.194, Lemma 5]{Suprunenko}). 
Thus, $N_\ell$ is abelian  and it has a natural structure as an $\F_\ell$-vector space. 
Further, ${\rm dim}_{\F_\ell}(N_\ell)$ is less than $n-1$ since $\F^\times_{2^m}$ is cyclic. 

Let $\tau$ be one of any generating set of  $H\simeq S_{n+1}$.  
Let $g$ be an element of $G$ which maps to $\tau$ by restriction. 
Let $K:=\langle N, g \rangle$ and $K_\ell$ be a Sylow $\ell$-subgroup 
of $K$. Let $S_\ell(K)$ (resp. $S_\ell(N)$) be the set of all 
Sylow $\ell$-groups of $K$ (resp. $N$). The restriction map $S_\ell(K)\mapsto S_\ell(N)$ 
is bijective and hence the natural map ${\rm Nor}_K(N_\ell)\lra K/N$ induces 
a bijection ${\rm Nor}_K(N_\ell)/{\rm Nor}_N(N_\ell)\lra K/N=\langle \tau \rangle$ where 
${\rm Nor}_A(B)$ stands for the normalizer of $B$ in $A$ for groups $B\subset A$.  
Thus, if we put $G_\ell={\rm Nor}_G(N_\ell)$, then $G_\ell$ contains $N$ as a normal subgroup 
so that $G_\ell/N\simeq H\simeq S_{n+1}$. 
Thus, $S_{n+1}$ acts $\F_\ell$-linearly on 
the $\F_\ell$-vector space $N_\ell$ of dimension strictly less than $n$. 
By \cite[Theorem 1.1]{Wagner-odd}, this action has to be trivial. 
Since $N_\ell$ is abelian, the restriction $\br_{t,2}|N_\ell$ decompose into one-dimensional spaces 
over $\bF_2$ and $S_{n+1}$ acts non-trivially on some component unless $N_\ell$. 
Obviously, this is impossible and hence $N_\ell$ is trivial.

Thus, $N$ is a 2-group. By \cite[p.184, Corollary 1]{Suprunenko}), 
$N$ is conjugate to a subgroup of upper Borel subgroup of $\SL_n(\F_2)$. 
Let $N_i=[N_{i-1},N_{i-1}]$ be the derived series with $N_0=N$. Then, $N_n$ is trivial.  
Note that the action of $S_{n+1}$ on $N$ induces one of each $N_i$. 
Consider the action of $S_{n+1}$ on $\F_2$-vector space 
 $N_{n-1}=N_{n-1}/N_n$ of dimension strictly less than $n$. 
This action has to be trivial by \cite[Theorem 1.1]{Wagner}. 
The same argument as above shows $N_{n-1}$ has to be trivial. 
By induction on $n-i$, we see that $N$ is trivial.  
Hence we have $K_{f_t}=K_{\br_{t,2}}$.    

The last claim is follows from \cite[Theorem 1.1]{Wagner} again.

\section{Galois representations I:Basics}\label{basics}
\subsection{Notation}We refer to \cite[Section 1]{BGGT} or \cite[Section 2]{Th2}. 
Let $n\ge 1$ be  an integer and define $\cG_n:=(GL_n\times GL_1)\rtimes \{1,j\}$ with 
$j^2=1$. 
The semi-direct product is defined by 
$$j(g,\mu)=(\mu{}^t g^{-1},\mu)j.$$
This is a non-connected algebraic group scheme over $\Z$.  
Let $\cG^0_n=GL_n\times GL_1$ be the connected component of $\cG_n$.
There is a character $\nu:\cG_n\lra GL_1$ sending 
$(g,\mu)$ to $\mu$ and $j$ to $-1$. 
Let $\G$ be a group with a normal subgroup $\Delta$ of index 2. Let $A$ be a commutative ring.  
For each homomorphism $r:\G\lra \cG_n(A)$ with $\Delta=r^{-1}(\cG^0_n)$, 
we define $\breve{r}:\Delta\lra \GL_n(A)$ by the composition of 
$$\Delta\stackrel{r|_\Delta}{\lra}\GL_n(A)\times A^\times\lra \GL_n(A)$$ 
where the latter map is the natural projection. 

For each homomorphism $r:\G\lra \GSp_{2n}(A)$ with the multiplier $\mu:=\nu \circ r$, 
we define 
\begin{equation}\label{extension}
\hat{r}_\Delta:\G\lra \cG_{2n}(A),\ g\mapsto 
\left\{\begin{array}{ll}
(r(g),\mu(g)) & (g\in \Delta) \\
(r(g)J^{-1}_{2n},-\mu(g))j & (g\in \G\setminus\Delta)
\end{array}\right.
\end{equation}
Then, $\Delta=\hat{r}^{-1}_\Delta(\cG^0_n(A))$ and $\nu\circ \hat{r}_\Delta=\mu$. 
Further, $\breve{\hat{r}}_\Delta=r|_{\Delta}$.

\subsection{Terminology}
Let $\ell$ be a rational prime and fix an isomorphism $\iota:\bQ_\ell\stackrel{\sim}{\lra}\C$. 
Let $F$ be a CM field or a totally real field and $F^+$ its maximal totally real subfield 
so that $F = F^+$ if $F$ is totally real. For each infinite place $v$, we denote by $c_v$ 
a complex conjugation in $G_F$ corresponding to $v$ so that $c_v=1$ if $F$ is a CM field.  

\begin{dfn}\label{various} 
 A polarized $\ell$-adic representation of $G_F$ is a couple $(r,\mu)$ of 
continuous homomorphisms  $r:G_F\lra \GL_n(\bQ_\ell)$ and $\mu:G_{F^+}\lra \bQ^\times_\ell$ 
which satisfy the following conditions: 
for some $($any$)$ infinite place $v$ of $F$, there exist $\ve_v\in\{\pm \}$ and 
a perfect bilinear pairing $\langle \ast,\ast \rangle_v:\bQ^n_\ell\times\bQ^n_\ell\lra \bQ_\ell$ 
such that 
\begin{enumerate}
\item $\ve_v=-\mu(c_v)$ if $F$ is CM.  
\item $\langle x,y \rangle_v=\ve_v \langle y,x \rangle_v$,
\item $\langle r(g)x,r(c_v g c^{-1}_v)y \rangle_v=\mu(g) \langle x,y \rangle_v$ 
for any $g\in G_F$.
\end{enumerate}
A polarized mod $\ell$ representation $(\bar{r},\bar{\mu})$ consisting of 
continuous homomorphisms   
$\bar{r}:G_F\lra \GL_n(\bF_\ell)$ and $\bar{\mu}:G_F\lra\bF^\times_\ell$ is defined similarly. 

An $\ell$-adic representation $r:G_F\lra \GL_n(\bQ_\ell)$ is said to be polarized if 
there exists a continuous character $\mu:G_{F^+}\lra \bQ^\times_\ell$ such that $(r,\mu)$ 
is polarized in the above sense. 
It is similarly defined for a mod $\ell$ representation $\bar{r}:G_F\lra \GL_n(\bF_\ell)$. 
\end{dfn}

By \cite[Lemma 2.2]{Th2}, if $(\rho,\mu)$ with $\rho:G_F\lra \GL_n(\bQ_\ell)$ and $\mu:G_{F^+}\lra \bQ^\times_\ell$ is polarized, there exists 
$$r:G_{F^+}\lra \cG_n(\bQ_\ell)$$ such that $r|_{G_F}=\rho$, $\mu=\nu\circ r$ and 
$G_F=r^{-1}(\cG^0)$. 
We call $r$ an extension of $(\rho,\mu)$. A mod $\ell$ version is also defined similarly. 
For example, if $F=F^+$, then $\hat{r}:=\hat{r}_{G_F}$ defined in (\ref{extension}) is an extension of 
a continuous homomorphism $r:G_F\lra \GSp_{2n}(\bQ_\ell)$ with $\mu:=\nu\circ r$.

\begin{dfn}\label{odd} Keep the notation as above. 
\begin{enumerate}
\item A polarized $\ell$-adic  representation $(r,\mu)$ of $G_F$ 
is said to be odd at $v$ if $\ve_v=1$ for an infinite place $v$ of $F$. 
When $\ell>2$, the oddness for a polarized mod $\ell$ representation 
is similarly defined. 
\item  Assume $n$ is even. 
For an infinite place $v$ of $F$, a polarized mod $2$ representation $(\rho,\bar{\mu})$ said to be strongly residually odd at $v$ if $\bar{r}(c_v)$ is $\GL_n(\bF_2)$-conjugate to $(1_n,1)j$ for 
an extension $\bar{r}$ of $(\rho,\bar{\mu})$.   
Further, a polarized mod $2$ representation $(\br,\bar{\mu})$ said to 
be strongly residually totally odd if it is strongly residually odd at $v$ for any 
infinite place $v$ of $F$. If $\bar{r}$ arises from $\br:G_{F^+}\lra \GSp_{2n}(\bF_2)$ with 
$\bar{\mu}:=\nu\circ \bar{\rho}$, then the above condition is equivalent to 
that $\br(c_v)$ is alternating, hence the corresponding non-degenerate alternating pairing $B:\bF^{2n}_2\times \bF^{2n}_2\lra \bF_2$ satisfies $B(x,x)=0$ for any $x\in \bF^{2n}_2$  
$($cf. \cite[Section 5.1]{BCGP2}$)$.
\end{enumerate}
\end{dfn}

\begin{dfn}\label{regularity} 
\begin{enumerate}
\item An $\ell$-adic representation is algebraic if it is unramified at all but finitely many 
finite places of $F$ and it is de Rham at all primes of $F$ above $\ell$. 
\item  An algebraic $\ell$-adic representation $r:G_F\lra \GL_n(\bQ_\ell)$ is regular if 
the set ${\rm HT}_{\tau}$ consisting of the Hodge-Tate weights of 
$r|_{G_{F_v}}$ at each finite place $v$ above $\ell$ has 
$n$ distinct elements for any $\tau:F\hookrightarrow \bQ_\ell$.  
\end{enumerate}
\end{dfn}

For notation around local and global Galois representations, $p$-adic Hodge theory, and 
the definition of automorphy for regular algebraic Galois representations, we refer to \cite[Section 1 and Section 2]{BGGT}. 
In particular, the reader should not confuse ``$\sim$'' (which means the connectivity \cite[p.524 for $\ell$-adic cae 
and p.530 for $p$-adic case]{BGGT}) with the equivalence of the Galois representations.

\section{Galois representations II}\label{GRII}
\subsection{Potential automorphy} 
In this section, we will prove a potentially automorphy theorem for a family of 
Galois representations by using the Dwork family. This result is an extension of 
\cite[p.546, Theorem 3.1.2]{BGGT}. This plays an important role to switch the 
residual characteristic. 

\begin{thm}\label{pot-auto} Suppose that 
\begin{itemize}
\item $F/F_0$ is a finite, Galois extension of totally real fields;
\item $r\ge 1$ is an integer;
\item  $n_1\ge 4$ is even integer, $\ell_1=2$, and $\iota_1:\bQ_{\ell_1}\stackrel{\sim}{\lra}\C$;
\item if $r\ge 2$, for each $i$ with $2\le i\le r$,  $n_i$ is a positive even integer,  
$\ell_i$ is an odd rational prime, and $\iota_i:\bQ_{\ell_i}\stackrel{\sim}{\lra}\C$;
\item $F^{{\rm (avoid)}}/F$ is a finite Galois extension;
\item for $1\le i\le r$, $\bar{r}_i:G_F\lra \GSp_{n_i}(\bF_{\ell_i})$ is a mod 
$\ell_i$ Galois representation with $\nu\circ \bar{r}_i=\bar{\ve}^{1-n_i}_{\ell_i}$ 
$($In particular, $\bar{r}_i$ is totally odd if $i\ge 2)$;
\item For $i=1$,  $\bar{r}_1$ is strongly residually odd at some infinite place of $F$;
\item ${\rm Im}(\bar{r}_i)\cap \Sp_{n_i}(\bF_2)$ is contained in ${\rm MD}_{n_i}(k)$ for 
some finite extension $k/\F_2$. 
\end{itemize}
Then, there exist a finite totally real extension $F'/F$ and 
a regular algebraic, cuspidal, polarized automorphic representation $(\pi_i,\chi_i)$ of 
$\GL_{n_i}(\A_{F'})$ for each $1\le i\le r$ such that 
\begin{enumerate}
\item $F'/F_0$ is Galois;
\item $F'$ is linearly disjoint from $F^{{\rm (avoid)}}$ over $F$;
\item $(\bar{r}_{\ell_i,\iota_i}(\pi_i),\bar{r}_{\ell_i,\iota_i}(\chi_i)\bar{\ve}^{1-n_i}_{\ell_i})
\simeq (\bar{r}_i|_{G_{F'}},\bar{\ve}^{1-n_i}_{\ell_i})$ for each $1\le i\le r$ and $\iota_{i}:\bQ_{\ell_i}
\stackrel{\sim}{\lra}\C$;
\item $\pi_i$ is $\iota_i$-ordinary of weight 0 for each $1\le i\le r$ and $\iota_{i}:\bQ_{\ell_i}
\stackrel{\sim}{\lra}\C$. 
\end{enumerate} 
\end{thm}
\begin{proof}The proof here follows the proof of \cite[Theorem 3.1.2]{BGGT}. 
We may assume that ${\rm Im}(\bar{r}_i)\subset \GSp_{n_i}(\F^{(i)})$ 
for some finite extension $\F^{(i)}/\F_{\ell_i}$. By assumption, we may put $\F^{(1)}=k$. 

We choose a positive odd integer $N$ such that 
\begin{itemize}
\item $N$ is coprime to $\ds\prod_{i=1}^r \ell_i$;
\item $N>n_i+1$ for any $1\le i\le r$;
\item $N$ is not divisible by any rational prime which ramifies in $F^{{\rm (avoid)}}$;
\item for $i=1$, there exists a non-zero prime $\lambda_1$ of $\Z[\frac{1}{N},\zeta_N]^+$ 
above 2 such that $k_{\la_1}$ contains $k$. 
\item if $r\ge 2$, for each $2\le i\le r$,  
there exists a non-zero prime $\lambda_i$ of $\Z[\frac{1}{N},\zeta_N]^+$ 
above $\ell_i$ such that $k_{\la_i}$ contains $\F^{(i)}$. 
\end{itemize}
The third condition shows $\Q(\zeta_N)$ is linearly disjoint from $F^{{\rm (avoid)}}$ over $\Q$. 
Next, for each $1\le i\le r$, we choose a CM field $M_i$ such that 
$M_i/\Q$ is cyclic of order $n_i$ and unramified at all primes which ramify in $F^{{\rm (avoid)}}$. 
Let $\tau_i$ denote a generator of ${\rm Gal}(M_i/\Q)$. 
We choose a rational prime $q$ such that 
\begin{itemize}
\item $q$ splits completely in $\ds\prod_{i=1}^r M_i$; 
\item $q$ is unramified in $F(\zeta_{4N})$. 
\end{itemize} 
Choose a finite field $M'$ containing $\ds\prod_{i=1}^r M_i$. 
Choose a rational prime $\ell'$ such that 
 \begin{itemize}
\item $\ell'$ splits completely in $M'(\zeta_N)$ (In particular, $\ell'\equiv 1$ mod $N$);
\item $\ell'$ is unramified in $F$ (In particular, $\zeta_{\ell'}\not\in F$);
\item $\ell'\nmid 3qN\ds\prod_{i=1}^r\ell_i n_i$.
\end{itemize}
Choose a prime $\lambda'_{M'}$ of $M'$ above $\ell'$. 
Then, by the argument for the proof of \cite[Theorem 3.1.2]{BGGT}, we can choose 
an algebraic character $\theta_i:G_{M_i}\lra O^\times_{M',\lambda'_{M'}}=\Z^\times_{\ell'}$ 
such that $r'_i:={\rm Ind}^{G_\Q}_{G_{M_i}}\theta_i$ factors through 
$\GSp_{n_i}(\Z_{\ell'})$ with the multiplier $\ve^{1-n_i}_{\ell'}$. 
Further, it satisfies the following properties:
\begin{itemize}
\item $\bar{r}'_i$ is irreducible;
\item $\bar{r}'_i(G_{\Q(\zeta_{l'})})$ is  adequate;
\item  $\bar{r}'_i(G_{\Q(\zeta_{l'})})=\bar{r}'_i(F_{\Q(\zeta_{Nl'})})$. 
\end{itemize} 
Applying Section \ref{connected} for $n=n_i$ $(1\le i\le N)$, we have a finite etale covering 
$T_{\bar{r}_i\times \bar{r}_i}$ over $T_0/F(\zeta_N)^+$ for each $1\le i\le r$. 
Put $T:=\ds\prod_{i=1}^r T_{\bar{r}_i\times \bar{r}_i}$ and for each point $P$ of $T$, 
we denote by $t_i(P)$ the underlying point in $T_0/F(\zeta_N)^+$ of the $i$-th 
component of $P$. By assumption and Proposition \ref{conn}, $T$ is geometrically connected. 
For $1\le i\le r$ and each place $v$ of $F(\zeta_N)^+$ dividing $\infty \ell_i \ell'$, we define 
$\Omega_{v,i}$ as follows. 
If $v$ is an infinite place, then we define $\Omega_{v}\subset T(F(\zeta_N)^+_v)
\simeq T_0(\R)$ to be the inverse image of 
$\widetilde{\Omega}_v$ in Lemma \ref{infinite}-(2) under the map 
$T\stackrel{{\rm pr}_1}{\lra} T_1\lra T_0$. 
By \cite[Lemma 2.16-(2)]{Th2}, Lemma \ref{connected}-(1), (2), and the argument for \cite[p.549]{BGGT},  $\Omega_{v}$ is non-empty.  
For each $v|\ell_i$ $(1\le i\le r)$, we choose 
$\Omega_v:=\{P\in T(F(\zeta_N)^+_v)\ |\ \ord_v(t_i(P))<0\}$ while 
for each $v|\ell'$, 
$\Omega_v:=\{P\in T(F(\zeta_N)^+_v)\ |\ \ord_v(t_i(P))>0\,\ 1\le\forall i\le r\}$. 
Enlarging $F$ by a finite totally real extension if necessary, we may assume these open sets are non-empty. 
Applying Moret-Bailly's theorem \cite[Proposition 3.1.1]{BGGT}, 
there exist a finite totally real extension $F'/F(\zeta_N)^+$ and $P\in T(F')$ such that 
\begin{itemize}
\item $F'/F_0$ is a totally real Galois extension; 
\item for each finite place $v|\ell_i$ $(1\le i\le r)$ of $F'$, $\ord_v(t_i(P))<0$;
\item for each finite place $v|\ell'$  of $F'$, $\ord_v(t_i(P))>0$ for each $1\le i\le r$;
\item $F'$ is linearly disjoint from $F^{{\rm (avoid)}}$ over $F$;
\item for each $1\le i\le r$, $(V_{\la_i,t_i(P)}/\la_i)(\frac{N-1-n_i}{2})\simeq \bar{r}_i|_{G_F'}$ for some non-zero prime ideal $\la_i$ of $\O_F'$ above $\ell_i$;
\item for each $1\le i\le r$, $(V_{\la',t_i(P)}/\la_i)(\frac{N-1-n_i}{2})\simeq \bar{r}'_i|_{G_F'}$ for some non-zero prime ideal $\la'$ of $\O_F'$ above $\ell'$;
\end{itemize}
Summing up, we have obtained the following:
\begin{itemize}
\item $F'$ is linearly disjoint from $F^{{\rm (avoid)}}$ over $F$;
\item for each $1\le i\le r$, $(V_{n_i,\la_i,t_i(P)}/\la_i)(\frac{N-1-n_i}{2})\simeq \bar{r}_i|_{G_F'}$ for some non-zero prime ideal $\la_i$ of $\O_F'$ above $\ell_i$;
\item for each $1\le i\le r$, $(V_{n_i,\la',t_i(P)}/\la_i)(\frac{N-1-n_i}{2})\simeq \bar{r}'_i|_{G_F'}$ for some non-zero prime ideal $\la'$ of $\O_F'$ above $\ell'$;
\item for each $1\le i\le r$, $\bar{r}'_i(G_{F(\zeta_{\ell'})})$ is adequate;
\item $\zeta_{\ell'}\not\in F'$;
\item for each $1\le i\le r$, $V_{n_i,\la',t_i(P)}(\frac{N-1-n_i}{2})$ is 
good ordinary reduction at all primes above $\ell'$ 
(it follows from \cite[Lemma 5.3-(3)]{BGHT} since $\ell'\equiv 1$ mod $N$);
\item  for each $1\le i\le r$, 
$HT_\tau(V_{n_i,\la',t_i(P)}(\frac{N-1-n_i}{2}))=\{0,1,\ldots,n_i-1\}$ for any 
$\tau:F'\hookrightarrow \bQ_{\ell'}$;
\item for each $1\le i\le r$ and each finite place $v|\ell_i$ of $F'$, 
$$\iota {\rm WD}(V_{n_i,\la',t_i(P)}(\frac{N-1-n_i}{2})|G_{F'_v})\simeq {\rm rec}_{F'_v}
(\Sp_{n_i}(\phi_i))$$
for some unramified character$\phi$ of $F'^\times_v$ and for any 
isomorphism $\iota:\bQ_{\ell'}\stackrel{\sim}{\lra}\C$ 
(this fact follows from Lemma \ref{pnotl}-(2)). 
\end{itemize}
Now the argument goes as in \cite[p.550]{BGGT}, the $\ell'$-adic representation 
$(r'_i|_{G_{F'}},\ve^{1-n_i}_{\ell'})$ is automorphic of level potentially prime to $\ell'$ 
by Arthur-Clozel's automorphic induction. Here we use the fact that 
any $\ell'$-adic character which is de Rham at any finite place above $\ell'$ is 
also potentially crystalline at any finite place above $\ell'$. 
Since $\ell'$ split completely in $M'$, $r'|_{G_{F'_v}}$ is decomposed into the 
direct summand of 1-dimensional 
spaces. Thus, it is ordinary for any $v|\ell'$. 
By construction, $V_{n_i,\la',t_i(P)}(\frac{N-1-n_i}{2})|G_{F'}$ has 
the reduction $\bar{r}'_i|_{G_{F'}}$ is ordinary at any finite place above $\ell'$. 
Applying Theorem \cite[2.4.1]{BGGT}, 
$V_{n_i,\la',t_i(P)}(\frac{N-1-n_i}{2})|G_{F'}$ is $\iota$-ordinary automorphic. 
Let $\pi_i$ be the corresponding cuspidal representation of $\GL_{n_i}(\A_{F'})$. 
By looking at the Hodge-Tate weights, $\pi_i$ is of weight 0. Further, 
for each $1\le i\le r$ and each finite place $v|\ell_i$ of $F'$, $\pi_{i,v}$ is a unramified 
twost of the Steinberg representation. In particularm $\pi_i$ is $\iota_i$-ordinary 
for each $\iota_{i}:\bQ_{\ell'}\stackrel{\sim}{\lra}\C$. 
Thus, $r_{\ell_i,\iota_i}(\pi_i)$ is an $\iota_i$-ordinary automorphic lift of $\bar{r}_i|_{G_F'}$ 
as desired.  
\end{proof}

\begin{rmk}\label{OFA} 
Keep the notation in Theorem \ref{pot-auto}. 
Pick $t_0\in T_0(F)$. For each infinite place $v$ of $F$, fix a conjugacy class $C_{t_0,v}$ of the image of 
the complex conjugation $c_v$ under the Galois representation associated to $V_{\la,t_0}/\la$. 
Then, one can find a non-empty subgroup $\Omega_v(C_{t_0},{v})$ containing $t_0$ of $T_0(F_v)=T(\R)$ such that 
the image of the complex conjugation $c_v$ corresponding to $v$ under the Galois representation attached to $V_{\la,t}/\la$ belongs to $C_{t_0,v}$ for any 
$t\in \Omega_v(C_{t_0,v})$.  
By using these open sets, a proof of Theorem \ref{pot-auto} goes through verbatim when  
we may replace the seventh condition therein on the mod 2 Galois representation 
$\bar{r}_1:G_F\lra \GSp_{n_1}(\bF_2)$ with the following condition:
\begin{itemize}
\item for some $t_0\in T_0(F)$, $\bar{r}_1(c_v)$ belongs to $C_{t_0,v}$ for each infinite place $v$ of $F$.    
\end{itemize}  
Thus, eventually, we can remove the seventh condition in  Theorem \ref{pot-auto}. 
However, most important case is when the mod 2 Galois representations are strongly residually odd 
at some infinite place. To avoid any confusion, we do not restate it here.  
\end{rmk}

\subsection{A Khare-Wintenberger' type lifting theorem}
The following theorem is an analogue of Khare-Wintenberger' lifting theorem 
(\cite[Theorem 5.1]{KW1} and \cite[Section 8]{KW2}) for  
prescribed ramification data. 
In \cite[Section 3.2]{BGGT}, they handled for any dimension but for even residual characteristic case, the tensor trick does not work since the induced representation is not (weakly) adequate.  
However, in this paper, we only focus on even dimensional Galois representation and thus, we can 
detour the all arguments involving induced representations.   

Let $F$ be a CM field with the maximal totally real subfield $F^+$. 
Let $S$ be a finite set of finite places of $F^+$ which split in $F$ and 
suppose $S$ includes all places above $2$. For each $v\in S$, we choose 
a finite place $\tilde{v}$ of $F$ above $v$. Let $\widetilde{S}:=\{\tilde{v}\ |\ v\in S\}$.  
Let $$\br:G_{F^+}\lra \GSp_{2n}(\F_2)$$ be a continuous representation. 
Let $$\bar{r}:G_{F^+}\lra \cG_{2n}(\bF_2)$$ be the extension of $\br$ given by (\ref{extension}). 
Then, $\breve{\bar{r}}=\br|_{G_F}$ by definition. 
Let $\mu:G_{F^+}\lra \bQ^\times_2$ be a continuous character unramified outside $S$ 
such that $\mu(c_v)=-1$ for any complex conjugation $c_v$ and it is crystalline 
at any finite place above 2. 
By the results in \cite[p.600, A.2]{BGGT}, there is an integer $w\in \Z$ such that  
${\rm HT}_\tau(\mu)=\{w\}$ for each $\tau:F^+\hookrightarrow \bQ_2$. 
For each $\tau:F\hookrightarrow \bQ_2$,  
let ${\rm H}_\tau$ be a set of $n$-distinct integers such that ${\rm H}_{\tau\circ c}
=\{w-h\ |\ h\in {\rm H}_\tau\}$ for any complex conjugation $c$. 
For each $v\in S$ with $v\nmid 2$, let $\rho_v:G_{F_{\tilde{v}}}\lra \GL_{2n}(\O_{\bQ_2})$ 
be a lift of $\breve{\bar{r}}|_{G_{F_{\tilde{v}}}}=\br|_{G_{F_{\tilde{v}}}}$. 

\begin{thm}\label{KWtype}Keep the notation and the assumption being as above.
Aslo suppose the following conditions:
\begin{enumerate}
\item $\breve{\bar{r}}$ is nearly adequate;
\item $\rho$ is strongly residually odd at some infinite place of $F^+$; 
\item ${\rm Im}(\br)\cap \Sp_{2n}(\bF_2)$ is contained in $\MD_{n_i}(k)$ for some 
finite extension $k/\F_2$;  
\item $2n[F^+:\Q]\equiv 0$ mod 4;
\item For each finite place $u|2$ of $F$, the restriction 
$\breve{\bar{r}}|_{G_{F_u}}=\br|_{G_{F_u}}$ admits a lift $\rho_u:G_{F_u}\lra \GL_{2n}(\O_{\bQ_2})$ 
which is ordinary and crystalline with Hodge-Tate numbers  ${\rm H}_\tau$ for each 
$\tau:F_u\hookrightarrow \bQ_2$. 
\end{enumerate}
Then there is a lift $$r:G_{F^+}\lra \cG_{2n}(\O_{\bQ_2})$$
of $\bar{r}$ such that 
\begin{enumerate}
\item $\nu\circ r=\mu$;
\item if $u|2$, then $\breve{r}|_{G_{F_u}}$ is ordinary and crystalline with Hodge-Tate numbers  ${\rm H}_\tau$ for each 
$\tau:F_u\hookrightarrow \bQ_2$;
\item if $v\in S$ with $v\nmid 2$, then $\breve{r}|_{G_{F_{\tilde{v}}}}\sim \rho_v$;
\item $r$ is unramified outside $S$. 
\end{enumerate}
\end{thm}
\begin{proof}We mimic the latter part of the proof of \cite[Proposition 3.2.1]{BGGT}. 
Let $F_0/F^+$ be a totally imaginary extension linearly disjoint from 
$\overline{F}^{{\rm Ker}(\br)}$. 
By Theorem \ref{pot-auto} (for $F^{{\rm(aboid)}}:=\overline{F}^{{\rm Ker}(\br)}F_0$), there exist a Galois totally real field $F^+_1/F^+$ 
and a regular algebraic cuspidal polarized automorphic representation $(\pi_1,1)$ of 
$\GL_{2n}(\A_{F^+_1})$ such that 
\begin{itemize}
\item $F^+_1$ is linearly disjoint from $\overline{F}^{{\rm Ker}(\br)}$;
\item $\bar{r}_{2,\iota}(\pi_1)\simeq \rho|_{G_{F_1}}$ for each $\iota:\bQ_{2}\stackrel{\sim}{\lra}\C$;
\item $\pi_1$ is $\iota$-ordinary of weight 0 for each $\iota:\bQ_{2}\stackrel{\sim}{\lra}\C$. 
\end{itemize}
Put $F_1=F_0F^+_1$. The first property shows $\br(G_{F_1})$ is nearly adequate. 

Let $T'\supset S$ be a finite set of finite places of $F^+$ including all those above which 
$\pi_1$ and $F_1$ is ramified. Take a solvable extension $F^+_2/F^+$, linearly disjoint from 
$\overline{F}^{{\rm Ker}(\br|_{G_{F_1}})}_1$ such that all primes of $F^+_3:=F^+_1F^+_2$ above each element of $T'$ split in $F_3:=F_1F^+_2$. In particular, $F_3/F^+_3$ is unramified everywhere 
at finite places. Thus, we have 
$$\bar{r}_3:G_{F^+_3}\lra \cG_{2n}(\bF_2)$$
such that 
\begin{enumerate}
\item $\bar{r}_3(G_{F_3})=\br(G_{F_3})$ is nearly adequate;
\item $F_3/F^+_3$ is everywhere unramified;
\item  all places above 2 of $F^+_3$ and 
ramified places of the solvable base change $BC_{F_3/F_1}(\pi_1)$ split in $F_3$;
\item $\br|_{G_{F^+_3}}$ is strongly residually odd at some infinite place of $F^+_3$.   
\end{enumerate} 

For each $v\in S$ with $v\nmid 2$, let $\mathcal C_v$ be a component of 
$R^\Box_{\br|_{G_{F_{\tilde{v}}}}}\otimes\bQ_2$ containing $\rho_v$. 
Consider  $\O=\O_E$ with the residue field $\F$ for a finite extension $E/\Q_2$ in $\bQ_2$ such that $E$ contains the image of each embedding $F_3\hookrightarrow \bQ_2$, 
$\mathcal C_v$ is defined over $E$ for each $v\in S$, 
${\rm Im}(\br)\subset \GSp_{2n}(\F)$, and ${\rm Im}(\mu)\subset \O^\times$.  

For each $v\in S$ with $v\nmid 2$, let $\mathcal D_v$ be the deformation problem 
for $\br|_{G_{F_{\tilde{v}}}}$ corresponding to $\mathcal C_v$. 
For $u\in S$ with $u|2$, let $\mathcal D_v$ be the deformation problem 
consisting of all lifts of $\br|_{G_{F_u}}$ which factor through 
the ordinary deformation ring
$R^\Box_{\O,\br|_{G_{F_u}},\{{\rm H}_\tau\}_\tau,{\rm ord}}$ 
with the prescribed Hodge-Tate numbers (see \cite[Section 1.4]{BGGT} or \cite[Section 5.6]{BCGP2}). 
Let 
$$\mathcal S=(F/F^+,S,\widetilde{S},\{\Lambda_v\}_{v\in S},\bar{r},\mu,\{D_v\}_{v\in S})$$
where $\Lambda_v=\O$ if $v\nmid 2$ and $\Lambda_v=\Lambda_{\GL_{2n},v}$ 
defined in \cite[Section 5.6.5]{BCGP2}. 
We denote by $R^{{\rm univ}}_{\mathcal S}$ the universal deformation ring for 
the deformation datum $\mathcal{S}$. 

For each $v\in T$, we choose 
a finite place $\tilde{v}$ of $F_3$ above $v$. Let $\widetilde{T}:=\{\tilde{v}\ |\ v\in T\}$. 
For each $v\in T$ with $v\nmid 2$, let $\mathcal D_v={\rm Lift}^\Box_v$ for 
$\br|_{G_{F_{3,\tilde{v}}}}$. 
For each $v\in T$ with $v|2$, let $\mathcal D_v^\Delta$ be the 
deformation problem defined in \cite[Proposition 5.6]{BCGP2} 
for $\br|_{G_{F_{3,\tilde{v}}}}$. 
Put $T_2:=\{v\in T\ |\ v|2\}$. 
Let 
$$\mathcal S_3=(F_3,\bar{r},\O,\ve^{1-2n}_2, T,
\{\Lambda_v\}_{v\in S},\{D^\Delta_v\}_{v\in T_2}\cup \{D_v\}_{v\in T\setminus T_2})$$
be a global deformation problem defined in \cite[p.133]{BCGP2} (set $R=\emptyset$ therein)  
and denote by $R^{{\rm univ}}_{\mathcal S_3}$ its universal deformation ring. 
Applying \cite[Theorem 5.7.14]{BCGP2} for $\bar{r}_3$, we see that $R^{{\rm univ}}_{\mathcal S_3}$ is a finite 
$\Lambda=\otimes_{v\in T_2}\Lambda_v$-algebra. 
By \cite[Lemma 1.2.3]{BGGT}, the ring $R^{{\rm univ}}_{\mathcal S}$ is also 
a finite generated $\Lambda$-algebra. 
By a standard argument in \cite[Proposition 2.2.9 and Corollary 2.3.5]{CHT} 
and \cite[Proposition 5.5.6]{BCGP2} which is an extension of \cite[Proposition 2.21]{Th2}, 
every irreducible component of $R^{{\rm univ}}_{\mathcal S}$ has dimension 
at least ${\rm dim}_\O \Lambda=1+2n[F:\Q]$ (see also \cite[Lemma 5.7.7]{BCGP2}). 
Thus, there is a continuous ring homomorphism $R^{{\rm univ}}_{\mathcal S}\lra \bQ_2$ 
and it gives rise to a desired lift. 
\end{proof}

\section{Applications of \cite{TY} and \cite{BCGP2}}\label{App}
\subsection{The primitive part}
We shall apply the main theorems in \cite{TY} and \cite{BCGP2} to the Dwork quintic family. 
In this case, 
the residual automorphy of the primitive part 
has been studied in considerable detail in \cite{TY} and an ordinary automorphy lifting theorem 
has been established in \cite{BCGP2}. 
Assume $N=5$ and $n=4$ for the notation in Section \ref{DFII}. 
Let $F^+$ be a totally real field. For each $t\in F^+\setminus\{1\}$ with $t^5-1\neq 0$,    
let $\rho_{t,2}:G_{F^+}\lra \GSp_4(\Q_2)$ be the 2-adic Galois representation 
associated to $V_{t,2}=H^3_{\text{\'et}}(W_{t,\overline{F^+}},\Q_2)$ (recall (\ref{vt2})). 
Let $f_t(x)=4x^5-5tx^4+1$ be the trinomial defined in (\ref{tri}) with the decomposition field 
$F^+_{f_t}$ over $F^+$.  
By \cite[Theorem 1.1]{TY}, $\br_{t,2}\simeq {\rm Gal}(F^+_{f_t}/F^+)$ if $f_t$ is irreducible over 
$F^+$. 
As observed in \cite[Section 7,1]{TY} with the argument in \cite[Remark 5.2.8]{BCGP2},  
$\br_{t,2}$ is strongly residually totally odd if $\iota(t)<1$ for any $\iota:F^+\hookrightarrow \R$.  

To apply \cite[Theorem 5.7.14]{BCGP2}, we also need 
the following result (the third case is enough for our purpose):
\begin{prop}\label{PD}Keep the notation being as above. Assume $f_t$ is irreducible over $F^+$. 
 Let $v$ be a finite place of $F^+$ above 2. 
Then, the local Galois representation $\rho_{t,2}|_{G_{F^+_v}}$ satisfies the following properties:
\begin{enumerate}
\item If $F^+_v\simeq \Q_2$ and $\ord_v(t)=1$, then $\rho_{t,2}|_{G_{F^+_v}}$ is reducible and it is a non-trivial 
extension of two 2-dimensional crystalline representations. Further,   
 $$\rho_{t,2}|_{G_{F^+_v}}\sim {\rm Ind}^{G_{\Q_{2}}}_{G_{\Q_{2^2}}}\omega^{3}_2 \otimes \mu_{\zeta_4}\oplus 
 \Big( \ve_2\otimes\Big({\rm Ind}^{G_{\Q_{2}}}_{G_{\Q_{2^2}}}\omega_2 \otimes \mu_{\zeta_4}\Big)\Big)$$
where $\omega_2:G_{\Q_{2^2}}\lra G^{{\rm ab}}_{\Q_{2^2}}\simeq \Z^\times_{2^2}\times 
{\rm Gal}(\bF_2/\F_{2^2})\lra \Z^\times_{2^2}\hookrightarrow \O^\times_{\bQ_2}$ and 
$\mu_{\zeta_4}:G_{\Q_p}\lra \O^\times_{\bQ_2}$ is the unramified character sending ${\rm Frob}_2$ 
to $\zeta_4$. 
\item If $F^+_v\simeq \Q_2$ and $\ord_v(t)>1$, $\rho_{t,2}|_{G_{F^+_v}}$ is absolutely irreducible and 
$$\rho_{t,2}|_{G_{F^+_v}}\simeq {\rm Ind}^{G_{\Q_{2}}}_{G_{\Q_{2^4}}}\theta$$ 
where $\theta:G_{\Q_{2^4}}\lra \O^\times_{\bQ_2}$ is a crystalline character such that 
$\ds\bigcup_{\tau:\Q_{2^4}\hookrightarrow \bQ_2}{\rm HT}_{\tau}(\theta)=\{0,1,2,3\}$.  
\item If  $\ord_v(t-1)>0$, or $\ord_v(t)<0$, then $\rho_{t,2}|_{G_{F^+_v}}$ is potentially ordinary.   
\end{enumerate}
In either case,  the local Galois representation $\rho_{t,2}|_{G_{F^+_v}}$ is potentially diagonalizable. 
\end{prop}
\begin{proof}Let $K=F^+_v$ with the residue field $\F=\F_v$ for simplicity. Put $f=[\F:\F_2]$ 
and $\rho=\rho_{t,2}|_{G_K}$ for simplicity.
The case when $\ord_v(t)<0$ in the third claim follows from Lemma \ref{p=l}-(3). 

Assume $F^+_v=\Q_2$ and $\ord_v(t)>0$.  Now $K=F^+_v=\Q_2$, $\F=\F_v=\F_2$. 
Let $\mathcal{W}_0$ be a smooth model of $W_0$ over $\Z_2$. 
By using the comparison theorem \cite{BO}, we have a filtered $\vp$-module 
$$D:=D_{\rm crys}(\rho)\simeq H^3_{{\rm crys}}(\mathcal{W}_{0}\otimes_{\Z_2}\F_v,W(\F_v))
\otimes_{W(\F_v)}K$$
where $\varphi$ stands for the crystalline Frobenius and the filtration comes 
from Berthelot-Ogus's comparison theorem.
By \cite[Lemma 16, Lemma 23, and Theorem 24]{Gou}, 
$\det(t\cdot {\rm id}_D-\varphi|D)=t^4+2^6\in \Z_2[t]$ and $D$ is pure of slope $\frac{3}{2}$. 
It follows from this with the admissibility that  $\rho$ does not have any 1-dimensional subquotient over $\bQ_2$ as a filtered $\vp$-module. 
If $\rho$ is reducible over $\bQ_2$, then each irreducible constituent is of dimension 2. 
Write $\rho^{{\rm ss}}=\rho_1\oplus \rho_2$ where $\rho_1$ is 
crystalline of Hodge-Tate weights $\{1,2\}$ and $\rho_2$ is 
crystalline of Hodge-Tate weights $\{0,3\}$ (we also use $\rho^\vee(-3)\simeq \rho$). 
By \cite{GL}, $\rho_1\otimes \ve^{-1}_2$ is potentially diagonalizable and so is $\rho_1$. 
As for $\rho_2$, the eigen polynomial of the crystalline Frobenius on $D''$ is given by 
$t^2+a_2 t+2^3$ where $a_2=\zeta_8 2\sqrt{2}\alpha$ with $\alpha\in \{0,\pm 1\pm \zeta^2_8\}$. 
Then, $\ord_v(a_2)=\frac{3}{2}+\ord_v(\alpha)$ and 
$\ord_v(\alpha)=\infty$ or $\ord_v(\alpha)=\frac{1}{2}$. In either case, $\ord_v(a_2)>\frac{3}{2}$. 
By \cite[Theorem 1.1 and the existence of a family in Section 3, in particular, p.3182]{BL} 
with \cite[p.33, Proposition 3.2]{Breuil}, 
$\rho_2\sim \Big({\rm Ind}^{G_{\Q_2}}_{G_{\Q_{2^2}}}\omega^3_2\Big)\otimes\mu_{\sqrt{-1}}$. 
A similar argument shows $\rho_1\sim  \ve_2\otimes\Big({\rm Ind}^{G_{\Q_{2}}}_{G_{\Q_{2^2}}}\omega_2 \otimes \mu_{\zeta_4}\Big)$ 
as well. 
By the argument for the proof of \cite[Proposition 4.2]{Yam}, 
we see that $\rho\sim \rho_1\oplus \rho_2$ and the potentially diagonalizability follows from \cite[p.530, the property (5)]{BGGT}. 

When $\rho$ is irreducible over $\bQ_2$ (we may assume $\rho$ takes the coefficients in $E$), 
as \cite[p.33, Proposition 3.2]{Breuil} with \cite[the proof of Lemma 2.1.12]{GHLS}, it is easy to see $\rho_{t,2}|_{G_K}\simeq {\rm Ind}^{G_{\Q_2}}_{G_{\Q_{2^4}}}\theta$ where 
$\theta:G_{\Q_{2^4}}\lra \O^\times_E$ is some crystalline character. 
A key point is that $D=D_{{\rm crys}}(\rho)$ is pure of slope $\frac{3}{2}$. 
Since 
${\rm Ind}^{G_{\Q_2}}_{G_{\Q_{2^4}}}\theta$ 
is potentially diagonalizable by \cite[Lemma 1.4.3-(1)]{BGGT}, so is $\rho_{t,2}|_{G_K}$. 

Since $f_t$ is irreducible over $F^+$ by assumption, 
by \cite[Proposition 6.2 with Theorem 6.10]{TY}, the image of $\br_{t,2}|_{G_{\Q_2}}$ is isomorphic to 
the Galois group of $g_t(x):=x^5f_t(x^{-1})=x^5-5tx+4$ over $\Q_2$. When $\ord_v(t)=1$, $g_t$ is 
decomposed into two irreducible polynomials of degree 1 and 4 over $\Q_2$. In this case, 
we see that ${\rm Im}(\br_{t,2}|_{G_{\Q_2}})\simeq D_4$ and $\br_{t,2}(I_2)\simeq C_2\times C_2$ which lies in 
the unipotent radaical of the Siegel parabolic subgroup (cf. \cite[Section 2]{TY})  inside $\GSp_4(\F_2)$. 
On the other hand, when $\ord_v(t)>1$, hence $\ord_v(t)\ge 2$ (note that $\F^+_v=\Q_2$),  $g_t$ is 
irreducible over $\Q_2$ by Eisenstein-Dumas criterion. In this case, we see easily that  
${\rm Im}(\br_{t,2}|_{G_{\Q_2}})\simeq F_{20}\simeq C_5\rtimes C_4$ and $\br_{t,2}(I_2)\simeq C_5$. 

Thus, by observing the semisimplifiction of the mod 2 reduction, which case occurs is determined by the condition on 
$\ord_v(t)$. In fact, in the reducible case, the semisimplifiction of the mod 2 reduction is trivial and 
it never happens when $\ord_v(t)>1$. 

Finally, we assume the remaining condition that $\ord_v(t-1)>0$. We work on the Dwork quintic family. 
Let $W_K$ be the Weil-Deligne group of $K$.  
We view $Y_{t}$ as a projective flat 
scheme over $O_K$ whose generic fiber is denote by the same symbol. Let $\widetilde{Y}_t:=Y_t \otimes_{O_K}\F$ be the special fiber of $Y_t$. We may assume $\F$ includes 
a primitive fifth root $\zeta_5$ of unity by enlarging $K$ if necessary. 
By assumption, $Y_t$ has only quadratic singularity which sits on the special fiber and the singular locus of $\widetilde{Y}_t(\bF)$ is  explicitly given by 
$\Sigma:=\{[1:\zeta^{a_1}_5:\zeta^{a_2}_5:\zeta^{a_3}_5:\zeta^{a_4}_5]\in \mathbb{P}^4(\bF)\ |\ 
a_1,a_2,a_3,a_4\in \Z/5\Z,\  a_1+a_2+a_3+a_4=0\}$ 
(cf. \cite{Schoen} or \cite[Chapter 3]{Meyer}). 
Applying Picard-Lefschetz formula \cite[Expos\'e XV, Th\'eor\`eme .4]{SGA7II} 
(or \cite{Illusie}), for any fixed rational prime $\ell>2$, we have an exact sequence 
$$0\lra H^3_{\text{\'et}}(\widetilde{Y}_{t,\bF},\Q_\ell)\lra H^3_{\text{\'et}}(Y_{t,\overline{K}},\Q_\ell)\lra 
\sum_{x\in \Sigma}\Q_\ell(-2)\lra H^4_{\text{\'et}}(\widetilde{Y}_{t,\bF},\Q_\ell)\lra H^4_{\text{\'et}}(Y_{t,\overline{K}},\Q_\ell)\lra 0$$
such that the image of $H^3_{\text{\'et}}(\widetilde{Y}_{t,\bF},\Q_\ell)$ is contained in $H^3_{\text{\'et}}(Y_{t,\overline{K}},\Q_\ell)^{I_K}$. 
Notice that $H^4_{\text{\'et}}(Y_{t,\overline{K}},\Q_\ell)\simeq \Q_\ell(-2)$ and thus 
$H^4_{\text{\'et}}(\widetilde{Y}_{t,\bF},\Q_\ell)^{{\rm ss}}$ is pure of weight $-2$. 
By using this with \cite[p.134]{CORV}, we see that 
the Weil-Deligne represenation associated to $H^3_{\text{\'et}}(W_{t,\overline{K}},\Q_\ell)$ 
is isomorphic to $\chi_1\oplus \chi_2\oplus V$ or $\chi\otimes {\rm Sp}(2)\oplus V$ where 
each of $\chi_1,\chi_2,\chi$ is an unramified twist of 
a power of $\ve_2|_{W_K}$ and   
${\rm Sp}(2)$ is defined in \cite[(4.1.4)]{Tate}. Further, $V$ is an unramified representation of 
$W_K$ and the Frobenius eigen polynomial is given by $a_{2^f}$ where $a_{2^f}$ (resp. $a_2$) is 
$2^f$-th ($2$-nd) Fourier coefficient of a newform in $S_4(\G_0(25))$ (see \cite[p.134]{CORV}). 
Notice that $a_2=1$ and thus, $a_{2^f}\equiv (a_2)^f=1$ mod 2. Let $\pi_v$ be a unique 
admissible representation of $\GL_4(K)$ corresponding to $H^3_{\text{\'et}}(W_{t,\overline{K}},\Q_\ell)$. Since $H^3_{\text{\'et}}(W_{t,\overline{F^+}},\Q_\ell)$ is potentially automorphic 
by \cite{CHT}, replacing $F^+$ with a finite totally real extension, we can apply 
\cite[Theorem 1.1]{Ca} and \cite[p.538, Theorem 2.1.1-(6)]{BGGT} to 
conclude $\rho_{t,2}|_{G_K}$ is potentially ordinary. 
Thus, it is potentially diagonalizable by \cite[Lemma 1.4.3]{BGGT}. 
\end{proof}

\begin{thm}\label{autoDwork}Keep the notation being as above. 
Assume the Galois group of $f_t$ contains $A_5$ and $\iota(t)<1$ for any $\iota:F^+\hookrightarrow \R$. Further suppose that each finite place $v$ of $F^+$ above 2 satisfies 
$\ord_v(t)<0$ or $\ord_v(t-1)>0$. 
Then, $\rho_{t,2}$ is absolutely irreducible and it comes from a regular algebraic cuspidal 
automorphic representation  of $\GSp_4(\A_{F^+})$ of weight 0.  Further, in terms of classical language, one can find a holomorphic Hilbert-Siegel modular form $h$ of the paralell weight 
$\overbrace{((3,3),\ldots,(3,3))}^{[F^+:\Q]}$ on $\GSp_4(\A_{F^+})$ such that  
$$\rho_{t,2}\simeq \rho_{h,2}.$$
Here $\rho_{h,2}:G_\Q\lra \GSp_4(\bQ_2)$ is the 2-adic Galois 
representation attached to $h$ $($cf. \cite[Section 3]{Mok}$)$. 

Furthermore, $h$ is a symmetric cubic lift or genuine form in the sense of 
\cite[Section 2]{KWY}. 
In particular, the primitive part of 
$H^{3}_{\text{\'et}}(Y_{t,\overline{F^+}},\Q_\ell)$ is automorphic 
for any rational prime $\ell$ and it is also absolutely irreducible for   
the set of rational primes $\ell$ of Dirichlet density one.
\end{thm}
\begin{proof}By \cite[Theorem 1.2]{TY}, $\br_{t,2}$ is absolutely irreducible and 
so is $\rho_{t,2}$. 
By \cite[Theorem 1.2]{TY} again and by assumption, 
there exists at most quadratic, totally real extension $F^+_1/F^+$ (in fact $F^+_1=F^+(\sqrt{5(1-t^5)})$) such that 
$\br_{t,2}|_{G_{F^+_1}}$ satisfies 
\begin{enumerate}
\item $\br_{t,2}(G_{F^+_1})\simeq A_5$ and it factors through the absolutely 
irreducible representation $\theta_4$ over $\bF_2$ of $A_5$;
\item $\br_{t,2}|_{G_{F^+_1}}$ is strongly residually totally odd; 
\end{enumerate}
The first condition implies $\br_{t,2}|_{G_{F^+_1}}$ (and also $\rho_{t,2}|_{G_{F^+_1}}$) is absolutely irreducible and 
$\br_{t,2}(G_{F^+_1}))$ is nearly adequate by \cite[Lemma 8.2.9]{BCGP2}.  
By assumption with Proposition \ref{PD}-(3), $\rho_{t,2}$ is ordinary at 
any place of $F^+_1$ dividing 2. 

By argument for \cite[Section 4]{TY} and by using  \cite[Lemma 8.3.1]{BCGP2} 
(and also Proposition \ref{PD}-(3)),  
we can find a solvable totally imaginary solvable extension $M/F^+_1$  
and a regular algebraic self-dual cuspidal automorphic representation $\pi$ of $\GL_2(\A_{M})$ of weight 0 such that 
\begin{enumerate}
\item $\rho_{t,2}|_{G_M}$ is ordinary at any place of $M$ dividing $2$;
\item $\pi$ is $\iota$-ordinary at any place of $M$ dividing 2; 
\item $\br_{t,2}|_{G_{M}}\simeq {\rm Sym}^3(\bar{r}_{2,\iota}(\pi))$;
\item $\br_{t,2}(G_{M})=A_5$;
\item  $M/M^+$ is everywhere unramified. 
\item All places of $M^+$ above 2 split in $M$.  As do all places lying under a place 
at which $\pi$ is ramified.   
\end{enumerate}
Thus, all conditions of \cite[Theorem 5.7.14]{BCGP2} are fulfilled and  
$\rho|_{G_{M}}$ is automorphic. The claim follows from \cite[Lemma 2.2.2]{BGGT}.   

The latter claim follows from the condition on the image of $\br$ and \cite[Theorem 3.2]{CG}.
\end{proof}

\begin{rmk}\label{variousremark}
\begin{enumerate}
\item The assumption being strongly residually totally odd is equivalent to 
$\iota(t)<1$ for any $\iota:F^+\hookrightarrow \R$. 
This is necessary to conclude $\br_{t,2}$ is automorphic by using 
\cite[Theorem 1.2]{TY} or \cite[Lemma 8.3.1]{BCGP2}. 
\item A twisted version of the Dowrk quintic family has been studied in \cite[Section 8]{TY} 
and by using \cite[Theorem 8.1]{TY}, we have a similar result of Theorem \ref{autoDwork} 
for that family. 
\item In Theorem \ref{autoDwork}, if the Galois group of $f_t$ is $S_5$, then  $h$ can not be a symmetric cubic lift because of 
the classification of the projective images of the finite subgroups in $\GL_2(\bF_2)$ 
\cite[Theorem D]{Faber} 
$($or because of the minimal dimension of the irreducible representation of $S_5$ over $\bF_2$ \cite{Wagner}$)$. 
\end{enumerate}
\end{rmk}

\begin{cor}Assume $F^+=\Q$ and $t\in \Q\setminus\{0\}$ satisfies $t<1$. Assume 
$f_t$ is irreducible over $\Q$. Further assume $\ord_2(t)<0$ or $\ord_t(t-1)>0$. 
Then, $\rho_{t,2}:G_\Q\lra \GSp_4(\Q_2)$ comes from a regular algebraic cuspidal 
automorphic representation of $\GSp_4(\A_{\Q})$ of weight 0.  In terms of classical language, one can find a holomorphic Siegel modular form $h$ of the paralell weight 
$(3,3)$ on $\GSp_4(\A_{\Q})$ such that $\rho_{t,2}\simeq \rho_{h,2}$. 
Furthermore, $h$ is a genuine form in the sense of 
\cite[Section 2]{KWY} unless $t=0$ in which case $h$ is the automorphic induction of 
a certain Hecke character of $\A_{\Q(\zeta_5)}$. 
\end{cor}

\subsection{The non-parimitive part}\label{npp} 
Let $\ell$ be an arbitrary fixed rational prime $\ell$. 
Let $K$ be a number field. Recall  the Dwork quintic family 
\begin{equation}\label{dwork5}
Y:X^5_1+X^5_2+X^5_3+X^5_4+ X^5_5-5t X_1X_2X_3X_4X_5=0.
\end{equation} 
Assume $t\in K\setminus\{1\}$. In this section, we  will prove 
the strong Hasse-Weil conjecture for the middle cohomology of $Y_t$ under some condition. 
The primitive part is done in the previous section. Therefore, we may focus on the non-primitive part. If $t=0$,  by using \cite[p.104, Corollary 2.5]{KS}, 
$H^3_{\text{\'et}}(Y_{0,\bK},\Q_\ell)$ can be written as a direct summands of 
the induced representation 
of $\ell$-adic character attached to some algebraic Hecke characters on $\A_{K(\zeta_5)}$. 
Hence, the $L$-function $L(s,H^3_{\text{\'et}}(Y_{0,\bK},\Q_\ell))$ is extended 
holomorphically in $s$ to the whole space of complex numbers. 
Thus, we may assume $t\neq 0$. 
Then, by \cite[see p.478-479]{Gou},  as a $\Q_\ell[G_{K}]$-module, 
$$H^3_{\text{\'et}}(Y_{t,\overline{K}},\Q_\ell)^{{\rm ss}}\simeq 
H^3_{\text{\'et}}(W_{t,\overline{K}},\Q_\ell)^{{\rm ss}}
\oplus \bigoplus^{10}H^1_{\text{\'et}}(\mathcal A_{t,\overline{K}},\Q_\ell)(-1)
\oplus \bigoplus^{15}H^1_{\text{\'et}}(\mathcal B_{t,\overline{K}},\Q_\ell)(-1)$$
where $\mathcal A_t$ and $\mathcal B_t$ are projective smooth models over $K$ of 
affine models 
$$\mathcal A^\circ_t:y^5=x^2(1-x)^3(x-t^3)^2$$
and 
$$\mathcal B^\circ_t:y^5=x^2(1-x)^4(x-t^5).$$
Here we replace the parameter $\psi$ therein with our parameter $t$. 
The curves $\mathcal A_t$ and $\mathcal B_t$ have genus four. 

It is easy to see that each of there curve inherits an involution defined over $K$ as follows:
$$\iota_{1}:\mathcal A^\circ_t\lra \mathcal A^\circ_t,\ (x,y)\mapsto 
\Big(\frac{t^5}{ x}, \frac{t^4 (t^5 - x) (-1 + x)}{xy}\Big),$$
and 
$$\iota_{2}:\mathcal B^\circ_t\lra \mathcal B^\circ_t,\ (x,y)\mapsto 
\Big(\frac{t^5}{ x}, \frac{t^3 (t^5 - x) (-1 + x)}{xy}\Big),$$
respectively. These are naturally extended to involutions over $K$ on the 
projective smooth models and denote them by the same symbols.  
By \cite[see p.478-479]{Gou} and Faltings' theorem \cite{Fal}, it is easy to see that 
${\rm Jac}(\mathcal A_t)$ 
(resp. ${\rm Jac}(\mathcal B_t)$) is isogenous over $K$ 
to ${\rm Jac}(\mathcal A_t/\langle \iota_1 \rangle)^2$ (resp. 
${\rm Jac}(\mathcal B_t/\langle \iota_2 \rangle)^2$). 

For each $k=1,2$, put 
$$(X,Y)=\Big(x+\iota_k(x),\ \frac{y}{x}+\iota_k\Big(\frac{y}{x}\Big)\Big).$$ 
Then, $\mathcal A_t/\langle \iota_1 \rangle)$ has an affine model 
$$2 t^5 (1 + 2 t^5 + t^{10}) - (1 + 6 t^5 + 5 t^{10}) X + 
 2 (1 + 2 t^5) X^2 - X^3 +5 t^3 (1 + 2 t^5 + t^{10}) Y -$$ 
$$ 10 t^3 (1 + t^5) X Y + 5 t^3 X^2 Y - 5 t^4 (1 + t^5) Y^3 + 
 5 t^4 X Y^3 + t^5 Y^5=0$$
 and $\mathcal B_t/\langle \iota_2 \rangle)$ has an affine model 
$$2 t^5 (1 - 2 t^5 - 3 t^{10}) + 
 t^5 (1 + 10 t^5 + t^{10}) X - (1 + 4 t^5 + t^{10}) X^2 + X^3 + 
 5 t^6 (1 + 2 t^5 + t^{10}) Y -$$
 $$ 10 t^6 (1 + t^5) X Y + 5 t^6 X^2 Y + 
 5 t^8 (-1 - t^5) Y^3 + 5 t^8 X Y^3 + t^{10} Y^5=0.$$
 For the affine model of $\mathcal A_t/\langle \iota_1 \rangle$, 
 transforming $(X,Y)$ into $(X,Y)=(y^2 x + t^5 + 1, y)$ and then 
 proceeding as  
$$y \mapsto \frac{y}{x^3},\ y \mapsto \frac{y}{2} + \frac{1}{2}(5 t^3 x^2 + 5 t^4 x + t^5),$$
we have an affine model of genus 2
\begin{equation}\label{Ct}
C_t:y^2=P_t(x):= 4(t^5-1)x^5+ 25 t^6 x^4+ 50 t^7 x^3 + 35 t^8 x^2+ 10 t^9 x  +t^{10}.
\end{equation}
Note that the discriminant of $P_t(x)$ in $x$ is $2^8\cdot 5^5\cdot t^{40}(1-t^5)^2$. 
Similarly, as for the affine model of $\mathcal B_t/\langle \iota_2 \rangle)$, 
transforming $(X,Y)$ into $(X,Y)=(x y^2 + t^5 + 1, y)$ (which is slightly different from the former one) and then 
 proceeding as  
$$y \mapsto \frac{y}{x^3},\ y \mapsto \frac{y}{2}  + \frac{1}{2}(5 t^3 x^2 + 5 t^4 x + t^5),$$
we have the same affine model $C_t$ of genus 2.  
Summing up, we have proved the following result which is expected in 
\cite[p.480, after Remark 4]{Gou}:
\begin{prop}\label{AtBt}Keept the notation being as above. For each $t\in K\setminus\{0,1\}$, 
$${\rm Jac}(\mathcal A_t)\stackrel{K}{\sim} 
{\rm Jac}(\mathcal A_t/\langle \iota_1 \rangle)^2
\stackrel{K}{\sim}  {\rm Jac}(C_t)^2 
\stackrel{K}{\sim} 
{\rm Jac}(\mathcal B_t/\langle \iota_2 \rangle)^2
\stackrel{K}{\sim}  {\rm Jac}(\mathcal B_t).$$
In particular, we have a decomposition, 
\begin{equation}\label{dtotal}
H^3_{\text{\'et}}(Y_{t,\overline{K}},\Q_\ell)^{{\rm ss}}\simeq 
H^3_{\text{\'et}}(W_{t,\overline{K}},\Q_\ell)
\oplus \bigoplus^{50}H^1_{\text{\'et}}(C_{t,\overline{K}},\Q_\ell)(-1)
\end{equation}
as a $\Q_\ell[G_K]$-module. 
\end{prop}

Henceforth, we assume $K=F^+$ is a totally real field. 
Let us consider the complex multiplication $\iota:(x,y)\mapsto (x,\zeta_5 y)$ on each of 
$\mathcal A^\circ_t$ and $\mathcal B^\circ_t$. 
Then $\iota^\ast+(\iota^{-1})^\ast$ is compatible with $\iota_1$ and $\iota_2$ 
as an algebraic correspondence defined over $F^+(\sqrt{5})$ on $\mathcal A_t$ and $\mathcal B_t$ respectively. 
It follows from this that 
${\rm Jac}(\mathcal A_t/\langle \iota_1 \rangle)$, 
${\rm Jac}(\mathcal B_t/\langle \iota_1 \rangle)$, and  ${\rm Jac}(C_t)$ 
(which are isogenous over $F^+$ each other) are abelian surfaces over $F^+$ and 
if ${\rm Jac}(C_t)$ is $F^+(\sqrt{5})$-simple,  they are  abelian surfaces of GL$_2$-type over $F^+(\sqrt{5})$ with the real multiplication by $\Q(\sqrt{5})$ in the sense of \cite{Ho} or \cite{Ribet}. 

\begin{lem}\label{Pt}Assume $t\in F^+\setminus\{0,1\}$.  Let $F^+_{P_t}$ be the decomposition field of $P_t$ over $F^+$ with the 
Galois group $G_{P_t,F^+}$. Then it holds 
\begin{enumerate}
\item The Galois group $G_{P_t,F^+}$ is solvable.
\item If $P_t$ is irreducible over $F^+$, then $G_{P_t}$ is isomorphic to $C_5$, $D_{10}$, 
or $F_{20}$. 
\item If $\sqrt{5}\not\in F^+$ and $P_t$ is irreducible over $F^+$, then $G_{P_t,F^+}$ is isomorphic to $F_{20}$. 
\item If $P_t$ is irreducible over $F^+_1:=F^+(\sqrt{5})$ and $\iota(t)<1$ for 
some embedding $\iota:F^+_1\hookrightarrow \R$, then 
the Galois group $G_{P_t,F^+_1}$ of $P_t$ over $F^+_1$ is isomorphic to $D_{10}$. 
Further, under the identification $G_{P_t,F^+_1}\subset S_5$, an element of order 2 
corresponds to an element of type $(2,2)$. 
\item ${\rm Jac(C_t)}$ is simple over $F^+_1$ if $P_t$ is irreducible over $F^+$.  
\end{enumerate}
\end{lem}
\begin{proof}
The first claim follows from \cite[Theorem 1]{Du}. Here, we used the Mathematica version 12.1 to check it. 
The second claim follows from the classification of solvable subgroups inside $S_5$ whose 
orders are divisible by 5. The third claim follows from the criterion that 
$G_{P_t,F^+}$ is $C_5$ or $D_{10}$ if and only if 
the discriminant of $P_t$ is square. 

For the fourth claim, put 
\begin{equation}\label{Qt}
Q_t(x):=t^{-10}x^5 P_t(t x^{-1})=
x^5+10 x^4+35 x^3+50 x^2+25 x+4 - 4t^{-5}
\end{equation}
with the discriminant $2^8\cdot 5^5\cdot t^{-10}(1-t^{-5})^2$. 
Then, $Q'_t(x)=5 (x^2 + 3 x + 1) (x^2 + 5 x + 5)$ 
and the extreme values are given by $4(1-t^{-5})$ for the roots of $x^2 + 5 x + 5$ 
and $-4t^{-5}$ for the roots of $x^2 + 3 x + 1$. By assumption, $P_t$ contains four complex 
roots, none of which are real numbers. Therefore, the Galois group contains 
an element of type $(2,2)$ as an element of $S_5$. 
The claim follows from it with the previous argument. 

Finally, we note that $P_t$ is irreducible over $F^+$ if and only if  $P_t$ is irreducible over $F^+_1:=F^+(\sqrt{5})$. 
In fact, this follows from that $P_t$ is of odd degree.  
Suppose ${\rm Jac(C_t)}$ is not simple over $F^+_1$ so that it is isogenous over $F^+_1$ to 
the product of two elliptic curves over $F^+_1$. The mod 2 Galois representation 
attached to ${\rm Jac}(C_t)[2]$ contains an element of order 5 since $P_t$ is irreducible over $F^+_1$ 
as above. However, any mod 2 Galois representations arising from elliptic curves over $F^+_1$ can not have 
elements of order 5. The claim follows. 
\end{proof}

For each rational prime $\ell$, let $\rho_{C_t,\ell}:G_{F^+}\lra \GSp_4(\Z_\ell)$ be the $\ell$-adic Galois representation attached to $C_t$ and $\br_{C_t,\ell}$ be its reduction modulo $\ell$. 
Assume $P_t$ is irreducible over $F^+$. By Lemma \ref{Pt}-(5), ${\rm Jac}(C_t)$ is $F^+_1$-simple so that it is 
an $F^+_1$-simple abelian surface of $\GL_2$-type over $F^+_1$ with 
the real multiplication by $\Q(\sqrt{5})$. 
Thus, $V_\ell({\rm Jac}(C_t))$ is of rank 2 over $\Q_\ell\otimes_{\Q}\Q(\sqrt{5})$. 
Then, we have a decomposition  $\rho_{C_t,\ell}\otimes \Q(\sqrt{5})\simeq \ds\prod_{\la|\ell}
\rho_{C_t,\la}$ where $\la$ runs over all finite place of $\Q(\sqrt{5})$ above $\ell$ and 
$\rho_{C_t,\la}:G_{F^+}\lra \GL_2(\Q(\sqrt{5})_\la)$ is the corresponding Galois representation 
for $\la|\ell$  
(cf. see \cite[Section 3]{Ribet1} and note that the argument therein is easily extended to our case). 
We denote by $\br_{C_t,\la}$ the reduction of $\rho_{C_t,\la}$ modulo $\la$.

\begin{thm}\label{non-pri}Let $t\in F^+\setminus\{1\}$. 
Assume $P_t$  is irreducible over $F^+$ and $\iota(t)<1$ for any 
$\iota:F^+\hookrightarrow \R$. Further assume that $\ord_v(t)<-\frac{4}{5}$ 
for each finite place $v$ of $F^+$ above 2.
Then, for any rational prime $\ell$, the $L$-function $L(s,\rho_{C_t,\ell})$ is extended to the whole space of $\C$ as an entire function in $s$.  
\end{thm}
\begin{proof}We may prove the claim for $\ell=2$. 
Let $\la$ be the unique finite place of $\Q(\sqrt{5})$ above 2. Put $F^+_1:=F^+(\sqrt{5})$.  
By Lemma \ref{Pt}-(5), ${\rm Jac}(C_t)$  is $F^+_1$-simple and of $GL_2$-type  
as discussed before, we have the $\la$-adic Galois representation 
$$\rho_{C_t,\la}:G_{F^+}\lra \GL_2(\Q_{2^2})$$
We denote by $\br_{C_t,\la}:G_{F^+_1}\lra \GL_2(\F_4)$ its reduction modulo $\la$. 
By Lemma \ref{Pt}-(3), the image of $\br_{C_t,\la}$ is isomorphic to $D_{10}$. 
Further, by Lemma \ref{Pt}-(4),  $\br_{C_t,\la}$ is strongly residually odd, thus it is also absolutely irreducible. 
Notice that the extension $L/F^+$ corresponding to 
 $\Z/5\Z\subset D_{10}={\rm Im}(\rho_{C_t,\la}|_{G_{F^+_1}})$ is obviously non-CM. 

Let $v$ be a finite place of $F^+$ above 2. By assumption on $v$, 
the reduction type of $C_t$ at $v$ is of type (VII) by \cite[Th\'eor\`eme 1]{Liu} 
(this can be checked easily by computing invariants in \cite{Liu} for $y^2=Q_t(x)$). 
By the argument for \cite[p.215, line 12-16]{Liu}, $C_t$ is potentially semi-stable. 
More precisely, the special fiber of the N\'eron model over ${\rm Spec}\hspace{0.5mm}\Z^{{\rm ur}}_2$ of 
${\rm Jac}(C_t)$ has $\mathbb{G}^2_m/\bF_2$ as the connected component. 
Since $\br_{C_t,\la}$ is potential automorphy (cf. \cite{Ta1}) and by using the local-global compatibility 
(cf. \cite{Skinner}), we see that 
 $\rho_{C_t,\la}|_{G_{F^+_v}}$ is potentially ordinary. 
 Therefore, we can find a totally 
 real solvable extension $M/F^+$  such that 
 \begin{itemize}
 \item $M/F^+$ is Galois;
 \item $M$ is linearly disjoint from $\overline{F^+}^{{\rm Ker}(\br_{C_t,\la}}$ over $F^+_1$;
 \item $\rho_{C_t,\la}|_{G_M}$  is ordinary.
 \item The extension $ML/M$ corresponds to 
 $\Z/5\Z\subset D_{10}={\rm Im}(\rho_{C_t,\la}|_{G_{M}})$ and it is also non-CM .  
 \end{itemize} 
 Thus, we can apply \cite[p.1237, THEOREM]{Allen} 
 to conclude automorphy of $(\rho_{C_t,\la}|_{G_M})$.
The claim follows by applying \cite[p.215, Theorem 6.2]{AC}. 
\end{proof}

Combining this result with Theorem \ref{autoDwork}, we obtain the following:
\begin{thm}\label{HW}Let $F^+$ be a totally real field and $t\in F^+\setminus\{1\}$. 
Suppose the following conditions:
\begin{enumerate}
\item  $\ord_v(t)<-\frac{4}{5}$ for each finte place $v$ of $F^+$ above 2: 
 \item  the Galois group of $f_t$ contains $A_5$;
 \item $\iota(t)<1$ for any $\iota:F^+\hookrightarrow \R$;
\item   $P_t$ is irreducible over $F^+$;
\end{enumerate}
Then, the strong Hasse-Weil conjecture for the middle cohomology of  $Y_t$ is true. Hence, for any rational prime $\ell$, the $L$-function 
$L(s,H^3_{\text{\'et}}(Y_{t,\overline{F^+}},\Q_\ell))$ is extended to the whole space 
of the complex numbers as an entire function in $s$. 
\end{thm}

\begin{rmk}\label{irredPt}Recall the polynomial $Q_t$ in $($\ref{Qt}$)$. 
When we view $t$ as a parameter, the affine curve over $\Q$ defined by $f(x,t):=Q_t(x)=0$ has 
an model of hyperelliptic curve of genus 2. In fact, by substituting 
$$(x,t)\mapsto (X,Y):=\Bigg(
\frac{t^2(1 + 3 x + x^2)}{2}, \frac{5t^5 (1 + 6 x + 11 x^2 + 6 x^3 + x^4 - 
   2 t^{-5})}{2}
\Bigg),$$
we have an hyperelliptic model $D_1:Y^2=5(1+8X^5)$. For a number field $K$, if $Q_t(x)$ has a factor of degree 1 
over $K$ for some $t\in K$ with $t^5-1\neq 0$, then it yields a $K$-rational point on $D_1$.  
Suppose $f(x,t)=(x^2+a_1x+a_2)(x^3+a_3x^2+a_4x+a_5)$ over $\Z[t^{-1}](\subset \Q(t))$. 
By deleting $a_2,a_3,a_4,a_5$, we the equation $h(t,a_1)=0$ which yields a union of two 
singular geometrically irreducible curves $D^+_2$ and $D^-_2$ of geometric 
genus 2 over $\Q(\sqrt{5})$ $($hence the curve defined by $h$ is irreducible over $\Q$ but 
reducible over $\Q(\sqrt{5}))$. In fact, 
$h(x+3,\frac{y}{2})=f^+_2(x,y)f^-_2(x,y)$ where 
$$f^{\pm}_2(x,y):=20 a (-2 + x)^2 (-1 + x) x^2 + 
 4 x^2 (-50 + 90 x - 55 x^2 + 11 x^3) + y^5,\ a=\pm\sqrt{5}$$
 and $D^\pm_2$ is defined by $f^\pm_2(x,y)=0$. 
Substituting 
$$(X,Y):=\Bigg(\frac{x (-5 + a + 2 x)}{2 y^2}, -\frac{
 5 (240 x^2 - 160 x^3 + 32 a x^3 + 32 x^4 - 7 y^5 + 
    a (-80  x^2 + 32  x^3 + 3  y^5))}{64 y^5}\Bigg),$$
$D^\pm_2$ is birationally equivalent over $\Q(\sqrt{5})$ to 
$\widetilde{D}^\pm_2:Y^2=\ds\frac{1}{8}(5a + 15)X^5 +\frac{ 1}{2048}(-105a +235)$. 
For a given number field $K$, if $P_t$ is reducible over $K(\sqrt{5})$, then $t$ gives rise to a rational point in  
$D_1(K(\sqrt{5}))\cup \widetilde{D}^+_2(K(\sqrt{5}))\cup \widetilde{D}^-_2(K(\sqrt{5}))$.  
It follows from this with Falting's theorem \cite{Fal} that $P_t$ is irreducible over $K(\sqrt{5})$ for all but finitely many $t\in K(\sqrt{5})\setminus\{0,1\}$. 

It is an interesting problem to determine rational points of these curves over 
$\Q(\sqrt{5})$. 
\end{rmk}

\section{Appendix A }
Recall the notation in Section \ref{DFII}. Let $K/\Q$ be a finite extension. 
Let $t\in K$ with $t^7-1\neq 0$.  
 In this appendix, we determine the residual image of 
$\br_{t,2}:G_K\lra \GSp_6(\F_2)=\Sp_6(\F_2)$  (when $n=6$) by using Proposition \ref{concrete1} and 
Shioda's theory for 28 bitangent lines of smooth quartic plane curves \cite{Shioda}. 
We introduce the plane quartic defined by  
\begin{equation}\label{quartic}
C_t:4xy^3 + x^3z - 7txy^2z + 2yz^3=0.
\end{equation}
It is smooth if and only if $t^7-1\neq 0$.  
Let $\br_{C_t,2}:G_K\lra {\rm Aut}_{\F_2}({\rm Jac}(C_t)[2](\overline{K}))$ be the 
mod 2 representation associated to $C_t$. 
\begin{prop}\label{plane}Let $t\in K$ with $t^7-1\neq 0$. 
Then, $(\br_{t,2})^{{\rm ss}}\simeq (\br_{C_t,2})^{{\rm ss}}$.
\end{prop}
\begin{proof}Let $v\nmid 2\cdot 7$ be a finite place of $K$ satisfying $\ord_v(t^7-1)=0$.  
Then, both of $\br_{t,2}$ 
and $\br_{C_t,2}$ are unramified at $v$. Put $q=q_v=\sharp O_K/v$. 
Let $\tau=(1234)(56)\in S_6$ act on $Z_t$ by permutation of the coordinates. 
Then, it is easy to see that the cardinality of the $\tau$-orbit  
$(x_1,x_2,x_3,x_4,x_5,x_6)\in Z_t(\F_q)$ is even unless $x:=x_1=x_2=x_3=x_4$ and $y=y_1=y_2$. 
Thus, we have 
$$\sharp Z_t(\F_q)\equiv \sharp\Big\{(x,y)\in (\F^\times_q)^2\ \Big|\ 
4x+2y+\frac{1}{x^4y^2}-7t=0 \Big\}\ {\rm mod}\ 2.$$
Replacing $(x,y)$ with $(x,\frac{1}{xy})$ and homogenizing as $(x,y)\mapsto (\frac{x}{z},\frac{y}{z})$, 
we see that $\sharp Z_t(\F_q)\equiv 
\sharp\Big\{(x,y)\in (\F^\times_q)^2\ \Big|\ 
4x+2y+\frac{1}{x^4y^2}-7t=0 \Big\}=1+\sharp C_t(\F_q)$ mod 2. 
Applying Lemma \ref{ZZ} and Lemma \ref{YW}, it holds that  
$\sharp W_t(\F_q)\equiv \sharp\overline{Z}_t(\F_q)  \equiv  1+\sharp Z_t(\F_q)$ mod 2. Thus, we have  
${\rm tr}(\br_{t,2})={\rm tr}(\br_{C_t,2})$. The claim follows from the Chebotarev density theorem.  
\end{proof} 

Thus, we may consider $\br_{C_t,2}$ to understand $\br_{t,2}$. 
We replace $[x:y:z]$ with $[x_1:\frac{x_0}{2}:x_2]$ so that 
the definition equation becomes 
$$\frac{1}{2} x_0^3 x_1 - \frac{7}{4} k x_0^2 x_1 x_2 + x_1^3 x_2 + x_0 x_2^3=0.$$
By using Shioda's theory \cite[Section 3, Theorem 8]{Shioda}, 
$K_{\br_{C_t,2}}:=\overline{K}^{{\rm Ker}(\br_{C_t,2})}$ is given by 
the decomposition field of the polynomial $\Psi\in \Z[t,x]$ defined by 
$$\Psi(t,x):=
16 + 2744 t^3 x^2 - 16352 t x^3 + 117649 t^6 x^4 - 172872 t^4 x^5 + 
 512344 t^2 x^6 + 2676352 x^7$$
 $$ - 4537890 t^5 x^8 + 39126696 t^3 x^9 - 
 26289088 t x^{10} + 43181985 t^4 x^{12} - 10682784 t^2 x^{13}$$
 $$ +  5597568 x^{14} + 11124176 t^3 x^{16} + 1104768 t x^{17} + 
 493920 t^2 x^{20} + 489984 x^{21} - 26880 t x^{24} + 256 x^{28}.$$ 
By Magma \cite{BCP}, when we view $t$ as a variable, we can check the affine curve over $\Q$ defined by  
$\Psi(t,x)=0\subset \mathbb{A}^2_{t,x}$ is absolutely irreducible (and it has geometric genus 6).  
 By \cite[p.175, Corollary 10.2.2-(b)]{FJ}, the splitting field  $M_\Psi$ of $\Psi$ over $\Q(t)$ 
is a regular extension over $\Q(t)$. By regularity, 
${\rm Gal}(K M_\Psi/K(t))\simeq  {\rm Gal}(M_\Psi/\Q(t))$ for any number field $K$. 
The specialization of $M_\Psi$ at $t=2$ has $S_8$-extension over $\Q$ by Magma, we see  
${\rm Gal}(M_\Psi/\Q(t))$ contains a subgroup which is isomorphic to $S_8$. Thus, there exists a Zariski dense set $H_K\subset K$ such that $\br_{C_t,2}$ is absolutely irreducible 
for any $t\in H_K$ (here we also use \cite[Theorem 1.1]{Wagner}).  
Thus, $\br_{t,2}\simeq \br_{C_t,2}$ for any $t\in H_K$ by Proposition \ref{plane}. 
Since $\MD_6(\F_2)=S_8$ by Proposition \ref{concrete1}, we see easily that   
${\rm Gal}(F M_\Psi/F(t))\simeq S_8$ for a sufficiently large number field $F$. By regularity, 
$S_8= {\rm Gal}(F M_\Psi/F(t))\simeq  {\rm Gal}(M_\Psi/\Q(t))$.  
Summing up, we have the following
\begin{thm}\label{image8}Keep the notation being as above. It holds that  
\begin{enumerate}
\item For each $t\in K$ with $t^7-1\neq 0$, it holds that ${\rm Im}(\br_{C_t,2})\simeq {\rm Gal}(K_{\Psi(t,x)}/K)$ and it is isomorphic to  
 a subgroup of $S_8$. 
 \item There exists a Zariski dense subset $H_K$ of $K$ such that 
 ${\rm Im}(\br_{C_t,2})\simeq S_8$ for any $t\in H_K$ with $t(t^7-1)\neq 0$. 
 \item Suppose ${\rm Im}(\br_{C_t,2})\simeq S_8$. Then, it factors through the 
 standard representation $\theta_8$ and thus, $\br_{C_t,2}$ is absolutely irreducible. 
In particular, $\br_{t,2}\simeq \br_{C_t,2}$. 
\end{enumerate}
\end{thm}
As a byproduct of the proof as above, we have  
\begin{cor} The extension $M_\Psi/\Q(t)$ is a regular $S_8$-extension.
\end{cor}

\section{Appendix B}\label{AppB}We keep the notation in Section \ref{DFII}. 
Let $n\ge 2$ be an even integer. The case $n=2$ is allowed.  
Let $K$ be a number field and pick $t\in K\setminus\{1\}$ such that $t^{n+1}-1\neq 0$. 
Let $\ell$ be a rational prime.  Let $\rho_{t,\ell}:G_K\lra \GSp_n(\Q_\ell)$ be the 
Galois representation attached to $H^{n-1}_{\text{\'et}}(W_{t,\overline{K}},\Q_\ell)$. 
By the excision  theorem in \'etale cohomology, we have 
$$H^{n-1}_{\text{\'et}, c}(Z_{t,\overline{K}},\Q_\ell) 
\stackrel{\alpha}{\lra} 
H^{n-1}_{\text{\'et}}(W_{t,\overline{K}},\Q_\ell)\stackrel{\beta}{\lra} 
H^{n-1}_{\text{\'et}}((W_t\setminus Z_t)_{\overline{K}},\Q_\ell).$$
Notice that $H^{n-1}_{\text{\'et}}(W_{t,\overline{K}},\Q_\ell)$ is pure of odd weight $n-1$. 
On the other hand, $W_t\setminus Z_t$ is described in terms of the toric resolution of 
the stratum of $\overline{Z}_f\setminus Z_f$ (see the proof of Lemma \ref{YW}) and thus,  
$H^{n-1}_{\text{\'et}}((W_t\setminus Z_t)_{\overline{K}},\Q_\ell)$ is a direct sum of 
Tate-twists by construction. 
It yields $\beta=0$ and thus, $\alpha$ is a surjection.
By \cite[Corollary 3.3.5, p.508]{Batyrev} and the comparison theorem, 
$H^{n-1}_{\text{\'et}, c}(Z_{t,\overline{K}},\Q_\ell) $ is of rank $d(\Delta)+n-1=2n$ 
(note that $d(\Delta):={\rm vol}_{N_\R}(\Delta)=n+1$). 
On other hand, by \cite[Theorem 3.4, p.358]{Batyrev0}, 
we have a surjection  
$ H^{n-1}_{\text{\'et}, c}(Z_{t,\overline{K}},\Q_\ell)\lra H^{n-1}_{\text{\'et}, c}(\mathbb{T}_{\overline{K}},\Q_\ell)(1-n)
\simeq \Q^{\oplus n}_\ell$ 
where $\mathbb{T}=\mathbb{G}^n_m$. Since 
$H^{n-1}_{\text{\'et}}(W_{t,\overline{K}},\Q_\ell)$ is pure of weight $n-1$, we have 
$$H^{n-1}_{\text{\'et}, c}(Z_{t,\overline{K}},\Q_\ell)^{{\rm ss}}\simeq 
H^{n-1}_{\text{\'et}}(W_{t,\overline{K}},\Q_\ell)^{{\rm ss}}\oplus \Q^{\oplus n}_\ell.$$
From \cite[Theorem 3.4, p.358]{Batyrev0} and the discussion immediately preceding it, 
we have $H^i_{\text{\'et}, c}(Z_{t,\overline{K}},\Q_\ell)=0$ if $i<n-1$ and 
$H^i_{\text{\'et}, c}(Z_{t,\overline{K}},\Q_\ell)\simeq 
H^i_{\text{\'et}, c}(\mathbb{G}^n_{m,\overline{K}},\Q_\ell)(-n-1)\simeq \Q^{\oplus c_i}_\ell(n-1-i)$, 
$c_i=\dbinom{n}{i+2-n}$ if $i>n-1$.  
Summing up, we have proved the following result:
\begin{thm}\label{pointCount}Keep the notation being as above. 
For each finite place $v\nmid \ell$ of $K$ satisfying $\ord_v(t^{n+1}-1)=0$, it holds that  
\begin{equation}\label{PointsZt}
\sharp Z_t(k_v)=\sum_{i=0}^{2(n-1)} (-1)^i {\rm tr}({\rm Frob}_v |H^i_{\text{\'et}, c}(Z_{t,\overline{K}},\Q_\ell))=\Bigg(\sum_{i=n}^{2n-2}c_i q^{i+1-n}_v\Bigg)-n-{\rm tr}(\rho_{t,\ell}({\rm Frob}_v)).
\end{equation}
\end{thm}

\section{Appendix C}\label{AppC}
Keep the notation in the previous section. 
Now, we work on $n=4$ and $\ell=3$ and the mod 3 residual representation $\br_{t,3}$ of 
$\rho_{t,3}$. 
Recall the equation of $Z_t\subset \mathbb{G}^4_m$ from (\ref{Zt}):
$$x_1+x_2+x_3+x_4+\frac{1}{x_1x_2x_3x_4}-5t=0.$$
Consider the automorphism $\sigma:(x_1,x_2,x_3,x_4)\mapsto (x_2,x_3,x_1,x_4)$ on $Z_f$ of 
order 3. 
Let $v$ be a finite place $v\nmid 15$ of $K$ satisfying $\ord_v(t^5-1)=0$. 
Then, the orbit of $\sigma$ on $Z_f(k_v)$ has the cardinality 3 unless 
$x_1=x_2=x_3$. Thus, substituting $x=x_1=x_2=x_3$ and $y=x_4$ into the above equation, we have 
$$\sharp Z_t(k_v)\equiv \sharp\Bigg\{ (x,y)\in (k^\times_v)^2\ \Bigg|\ 
3x+y+\frac{1}{x^3y}-5\bar{t}=0   \Bigg\}\ {\rm mod}\ 3$$
where $\bar{t}:=t$ mod $v$. 
Substituting $(x,y)\mapsto \Big(-\frac{1}{x},-\frac{y}{x}\Big)$ into the equation 
$3x+y+\frac{1}{x^3y}-5\bar{t}=0$, we have the equation $y^2+(3+5\bar{t} x)y=x^5$. 
Let $D_t$ be the hyperelliptic curve (globally) defined by 
$y^2+(3+5tx)y=x^5$ for $t\in K$ with $t^5-1\neq 0$. Since it has the Weierstrass point at 
the point at infinity (which is a $K$-rational point), 
$$\sharp D_t(k_v)=1+\sharp
\Bigg\{ (x,y)\in k_v^2\ \Bigg|\ y^2+(3+5\bar{t}x)y=x^5   \Bigg\}=
3+\sharp\Bigg\{ (x,y)\in (k^\times_v)^2\ \Bigg|\ y^2+(3+5\bar{t}x)y=x^5   \Bigg\}$$
$$\equiv \sharp Z_t(k_v)\ {\rm mod}\ 3.$$
Let $\br_{D_t,3}:G_K\lra \GSp_4(\F_3)$ be the mod 3 representation attached to $D_{t}$.  
Applying Theorem \ref{pointCount} for $n=4$, we have  
$${\rm tr}(\rho_{t,\ell}({\rm Frob}_v)\equiv q_v+1-\sharp D_t(k_v)\ {\rm mod}\ 3=
{\rm tr}(\rho_{D_t,\ell}({\rm Frob}_v).$$
By Chebotarev density theorem, we conclude 
$\br_{t,\ell}^{{\rm ss}}\simeq \br_{D_t,\ell}^{{\rm ss}}$.

\end{document}